\documentclass[preprint, 12pt, authoryear]{elsarticle}

\makeatletter
\def\ps@pprintTitle{%
 \let\@oddhead\@empty
 \let\@evenhead\@empty
 \def\@oddfoot{\footnotesize\itshape\hfill\today}%
 \let\@evenfoot\@oddfoot}
\makeatother

\usepackage[english]{babel}	
\usepackage[utf8x]{inputenc}		
\usepackage{enumerate}
\usepackage{textcomp}                
\usepackage{typearea}
\usepackage{pdfpages}
\usepackage[toc]{appendix}
\usepackage{todonotes}

\usepackage{natbib} 

\usepackage{amsfonts, amsmath, amssymb, amsthm, dsfont}
\usepackage{amsthm}			
\usepackage{thumbpdf}	
\allowdisplaybreaks

\usepackage{pgf, tikz}
\usepackage{tikz-cd}		
\usepackage{bbm}

\theoremstyle{definition}
\newtheorem{definition}{Definition}[section]
\theoremstyle{plain}
\newtheorem{theorem}[definition]{Theorem}
\newtheorem{proposition}[definition]{Proposition}
\newtheorem{lemma}[definition]{Lemma}
\newtheorem{cor}[definition]{Corollary}
\theoremstyle{remark}
\newtheorem{example}{Example}

\usepackage{graphicx}
\usepackage{hyperref}

\newcommand{\N}{\mathbb{N}}

\newcommand{\R}{\mathbb{R}}

\newcommand{\F}{\mathcal{F}}

\newcommand{\pconv}{\xrightarrow{P}}

\newcommand{\Var}{\operatorname{Var}}

\newcommand{\wconv}{\Rightarrow}

\journal{}

\begin{document}

\begin{frontmatter}



\title{Estimation of state-dependent jump activity and drift for Markovian semimartingales\footnote{To appear in: Journal of Statistical Planning and Inference. DOI:10.1016/j.jspi.2020.04.009}}


\author{Fabian Mies}

\address{RWTH Aachen University, Institute of Statistics\\
Wüllnerstraße 3, D-52062 Aachen, Germany\\
mies@stochastik.rwth-aachen.de}

\begin{abstract}
	The jump behavior of an infinitely active Itô semimartingale can be conveniently characterized by a jump activity index of Blumenthal-Getoor type, typically assumed to be constant in time. 
	We study Markovian semimartingales with a non-constant, state-dependent jump activity index and a non-vanishing continuous diffusion component. 
	A nonparametric estimator for the functional jump activity index is proposed and shown to be asymptotically normal under combined high-frequency and long-time-span asymptotics. 
	Furthermore, we propose a nonparametric drift estimator which is robust to symmetric jumps of infinite variance and infinite variation, and which attains the same asymptotic variance as for a continuous diffusion process. 
	Simulations demonstrate the finite sample behavior of our proposed estimators.
	The mathematical results are based on a novel uniform bound on the Markov generator of the jump diffusion.
\end{abstract}


\end{frontmatter}

\section{Introduction}
As data is available at finer temporal resolution, rather detailed models of stochastic processes become statistically tractable. Extending diffusion-based continuous time models, processes with some form of jump behavior have gained attention in the literature. While classical diffusion processes are fully described by their local volatility and drift, processes with jumps admit a greater flexibility, as the conditional jump behavior is summarized by a compensating measure, which is in general an infinite-dimensional object. Thus, statistical modeling of the jumps is essentially a nonparametric problem. In this paper, we specify the form of the local jump measure semiparametrically, while considering state-dependence of these parameters in a nonparametric framework. More precisely, we focus on the behavior of the infinite activity component of a jump diffusion in the form of an index of jump activity, similar to \cite{ait2009estimating}.

In particular, we develop estimators for a continuous-time, scalar Markov process of the type \begin{align}
	\label{eqn:SDEdef}X_t &= \int_0^t \mu(X_{s}) ds + \int_0^t \sigma(X_s) dB_s + J_t,\qquad t\geq 0, \\
	\nonumber J_t&=\int_0^t \int_{|c(X_{s-},z)|\leq 1} c(X_{s-},z) (N-\nu)(dz,ds) + \int_0^t \int_{|c(X_{s-},z)|> 1} c(X_{s-},z) {N}(dz,ds).
\end{align}
Here $\mu\colon\R \to\R$, $\sigma\colon\R \to [0,\infty)$, $c\colon\R^2\to \R$ are measurable functions, and ${N}$ is a Poisson counting measure on $\R\times [0,\infty)$ with intensity $\nu(dz) dt$ such that $\int (1\wedge c(X_{s-},z)^2)\nu(dz)<\infty$ almost surely, for each $s\geq 0$. Furthermore, we assume the jump term $J_t$ to be symmetric in the sense that for each $x$, the instantaneous compensating measure $\nu_t(dz)$ given by the image measure $c(X_{t-},\cdot)\circ\nu(dz)$ is symmetric. 
Throughout, we assume that the solution of \eqref{eqn:SDEdef} exists, which can be guaranteed, for example, by suitable Lipschitz conditions (see \cite{applebaum2009levy}). 
Hence, $X_t$ is a semimartingale. 
Detailed regularity assumptions on the functions $\mu(x)$ and $\sigma(x)$ will be imposed in the sequel.
Regarding the jump component, we impose the following conditions.
\begin{itemize}
	\item[\textbf{(J1)}] The compensating intensity measure given by the push-forward $c(x,z)\circ \nu(dz)$ admits a Lebesgue density $\rho(x,z)$, i.e.\ $c(x,z)\circ \nu(dz) = \rho(x,z)\,dz$ for each $x\in\R$. The jump density $\rho$ is symmetric w.r.t.\ $z$, i.e.\ $\rho(x,z)=\rho(x,-z)$. 
	\item[\textbf{(J2)}] There exists a $C_\rho>0$ and values $1<\tau<\alpha<2$ such that the jump density satisfies $|\frac{d^k}{dx^k}\rho(x,z)| \leq C_\rho \left( |z|^{-1-\alpha} \vee |z|^{-1-\tau} \right)$, for all $x\in\R$, and $k=0,1,2$.
\end{itemize}
Assumption (J2) ensures that the jump-activity index of $X_t$ is uniformly bounded away from $2$, that is $\int 1\wedge|c(x,z)|^\beta \, \nu(dz) = \int 1\wedge |z|^\beta\rho(x,z)\, dz < \infty$ for all $\alpha<\beta \leq 2$. Furthermore, (J2) bounds the tails of the intensity measure so that the Lévy process corresponding to the intensity $\rho(x,z)dz$ has finite first moments. The uniformity in $x$ requires a suitable form of smoothness. 

We consider the statistical setting of discrete observations $X_{t_i}$ for an equidistant grid $t_i=ih, i=1,\ldots, n+1$ of meshsize $h$, and we study joint high-frequency and ergodic asymptotics, i.e.\ $h=h_n\to 0$ and $T=nh_n \to \infty$ simultaneously.
Our interest lies in nonparametric estimation of the drift function $\mu(x)$, as well as the state-dependent behavior of the small jumps of $J_t$. 
In a more general framework, in particular, without imposing symmetry and a tail bound on the jump measure, \cite{Ueltzhofer2013} studies the jump dynamics by estimating $\rho(x,z)$ nonparametrically for finitely many values $z$.
In contrast to the mentioned study, we are interested in the detailed behavior of $\rho(x,z)$ as $|z|\to 0$.
More precisely, we will consider the case $\rho(x,z)\approx r(x)|z|^{-1-\alpha(x)}$ for $|z|\to 0$ in an appropriate sense to be made precise, and $\alpha(x) \in (0,2)$. 
Thus, the small jumps behave locally like an $\alpha(x)$-stable process.
This property is also referred to as locally-stable in the literature (e.g.\ \cite{masuda2016non}), and $\alpha(x)$ is the (spot) jump activity index of $X_t$. 
If $\rho(x,z) = r(x)|z|^{-1-\alpha(x)}$ for all $z$ exactly, the jump process is a stable-like process as considered by \cite{bass1988occupation,bass1988uniqueness}. 
For a L{\'e}vy process, $\alpha(x)\equiv\alpha$ is known as the Blumenthal--Getoor index \citep{blumenthal1961sample}. 
In section \ref{sec:stable-like}, we will construct nonparametric estimators of $\alpha(x)$ and $r(x)$ and derive their asymptotic distribution.
To the best of our knowledge, the only other statistical treatment of a non-constant jump activity index $\alpha(x)$ is due to \cite{todorov2017testing}, where a pure-jump process without a Brownian component is studied. 
The importance of the presence of the diffusion term is discussed below, together with further related work.

From a purely statistical perspective, the jump activity $\alpha$ resp.\ $\alpha(x)$ is of interest because it can be estimated in a pure high-frequency setting, keeping $T$ fixed. This case has been initially studied by \cite{ait2009estimating}, and later by \cite{jing2012jump}, \cite{bull2016near}. In contrast, estimation of the drift and the full Lévy measure requires observations over an increasing time span. We will always let $T\to\infty$, and the results derived in this paper reflect the latter distinction as the rate of convergence of our estimator for $\alpha(x)$ is faster than for the drift $\mu(x)$. Another motivation to study the jump activity index is raised by arbitrage theory in mathematical finance. Even for a Lévy process, different values of $\alpha$ and $r$ induce singular probability measures on path space \citep{ait2012identifying}, in accordance with the identifiability from high-frequency observations. Thus, any full specification of an equivalent pricing measure in continuous time needs to match the jump activity index of the process under the physical probability measure. 

Besides the jump component, the second object of interest to us is the drift $\mu(x)$.
For jump processes of the form \eqref{eqn:SDEdef}, the drift is in general not well defined as it depends on the specific truncation function $\psi$ used to define the compensated Poisson integral. 
Here, we employ the truncation function $\psi(z) = z \mathds{1}_{|z|\leq 1}$ to fix the value of the drift.
Nevertheless, the quantity $\mu(x)$ is somewhat universal if we impose the jumps to be symmetric.
Under condition (J1), any truncation satisfying $\psi(z)=\psi(-z)$ will yield the same drift value.
Another option to define a canonical drift term is to suppose that the jumps are summable, i.e.\ $\int |z| \rho(x,z)\, dz<\infty$, such that no truncation is necessary at all. 
The latter approach is used, e.g., by \cite{gloter2018jump}.
A third option is to assume that the law of the jumps is fully known to the statistician, as done by \cite{Amorino2018}, such that the effect of changing $\psi$ can be accounted for.
In this paper, we will assume symmetry when identifying the drift.
We propose a nonparametric estimator for $\mu(x)$ and derive its asymptotic distribution. 
By employing a jump-filtering technique similar to \cite{gloter2018jump}, we are able to derive the same asymptotic variance as in the diffusion case, while allowing for symmetric jumps of infinite variation and infinite variance as a nuisance.
Our estimator is similar to the nonparametric thresholded drift estimator of \cite{Mancini2011}, who require the jumps to be of finite activity.

The estimators proposed in this paper are based on smoothly truncated increments of the form $F_{n,t}=f(u_n(X_{t+h_n}-X_t))$ for a bounded smooth function $f$ and a sequence $u_n =o(h_n^{-1/2})$ as $h_n\to 0$. Here, we use the Markov property of $X$ to derive non-asymptotic approximations for the conditional moments of $F_{n,t}$. These approximations make use of new analytical bounds on the infinitesimal Markov generator of $X_t$. The transformed increments $F_{n,t}$ are localized by a Nadaraya-Watson kernel estimator with bandwidth $b\to 0$, only considering those increments where $|X_t - x| = \mathcal{O}(b)$. For estimation of the jump activity, we choose a suitable function $f$ as specified below and find that $E(f(u_n(X_{t+h_n}-X_t))|X_t=x)$ scales like $u_n^{\alpha(x)}$.
This scaling is exploited to estimate $\alpha(x)$. The exact construction of the estimator is described in section \ref{sec:stable-like}, where we also derive its asymptotic normality at rate $\sqrt{Tb} u_n^{\alpha(x)/2}$. The effective rate is thus determined by some required upper bounds on $u_n$, and we can achieve at least the rate $\sqrt{Tb}h^{-\frac{\alpha(x)}{8}}$. Using the derived approximation of conditional expectations, we are also able to construct a pointwise estimator for the drift $\mu(x)$ by choosing a nonlinear odd function $f$ of suitable form, and $u_n=1$. This tempered drift estimator is asymptotically normal at rate $\sqrt{Tb}$, even in the presence of infinite variation jumps with infinite variance. The asymptotic variance is found to be affected by the jump component. A jump-filtered drift estimator can be realized by letting $u_n\to\infty$ slowly, recovering the asymptotic variance of the continuous diffusion case. 

\subsection*{Related work}

For nonparametric drift estimation in the presence of jumps, \cite{johannes2004statistical} suggested to identify the process by a kernel estimator of the conditional polynomial moments. This idea is pursued rigorously by \cite{bandi2003functional}, who derive the asymptotic normality of conditional polynomial moment estimators as $T\to\infty, h_n\to 0$, in particular of the drift term. They restrict the process to have finitely many jumps with finite moments. An extension to local linear kernel estimators is given by \cite{hanif2012local}. If the jumps have infinite activity and are driven by a L{\'e}vy process, the drift can be estimated by a sieve regression \cite{schmisser2014non} or by a nonparametric kernel estimator \cite{Funke2015}, provided $T\to\infty$. The mentioned studies assume the increments of $X_t$ to have finite conditional variances. As a result, the rate of convergence of the kernel estimator with bandwidth $b\to 0$ is $\sqrt{Tb}$, matching the diffusion case without jumps. Jumps of infinite variance are studied by \cite{long2013nadaraya} for a stochastic differential equation $dX_t = \mu(X_t)dt + r(X_{t-}) dZ_t$ driven by a pure-jump $\alpha$-stable L{\'e}vy process $Z_t$. They show that a simple Nadaraya-Watson kernel estimator of the drift remains consistent, though the asymptotic distribution is no longer normal, and the rate of convergence drops to $(Tb)^{1-1/\alpha}$ for $\alpha\in(1,2)$. Note also that this estimator is consistent only for $\alpha>1$, as the heavy tailed increments of the stable distribution affect the local average. Similar results are obtained by \cite{lin2014local,wang2013local} for a local linear estimator. In an earlier study \cite{mies2018tempering}, we have shown that for the same $\alpha$-stable model with $\alpha\in(1,2]$, the Nadaraya-Watson estimator can be modified by considering the transformed increments $f(X_{t+h}-X_t)$ for an odd bounded function $f$. This modification, called tempering therein, yields an asymptotically normal estimator which recovers the Gaussian rate $\sqrt{Tb}$. Although allowing for heavy-tailed jumps of infinite variation, the model of \cite{long2013nadaraya} and \cite{mies2018tempering} is quite restrictive as the process contains no Brownian component and the driving L{\'e}vy process is fully specified. In contrast, the jump diffusion model \eqref{eqn:SDEdef} allows for more flexibility in the state-dependent behavior of the jumps, similar to \cite{bandi2003functional}, but without restrictions on activity and second moments. Our results show that the tempering approach recovers the rate $\sqrt{Tb}$ also in this more general setting. It should be noted that the assumption of symmetric jumps is important for the tempered estimator to be asymptotically unbiased.

Although matching the rate of convergence of the continuous case, the asymptotic variance of the tempered drift estimator is increased by the presence of the jumps.
In a parametric setting, the variance of the continuous case can be recovered by filtering the jumps \citep{gloter2018jump}. 
That is, jump filtering yields efficient parametric estimators. 
While \cite{gloter2018jump} requires the jumps to be of finite variation, \cite{Amorino2018} allow for arbitrary jumps, as long as the law of the jump process is known to the statistician.
In section \ref{sec:drift}, we show that jump filtering can also be applied in a nonparametric setting to eliminate the influence of the jumps asymptotically. 
In contrast to the described approaches, we consider the jump term as a nuisance and allow for infinite variation, while on the other hand requiring symmetry.

To study the jump behavior, a state-dependent jump compensator can be estimated, for example, by the already mentioned pointwise estimator of \cite{Ueltzhofer2013} for $\rho(x,z)$.
Another related line of research tests the jump compensator $\nu_s(dz)$ of an Itô-semimartingale for constancy, see \cite{Bucher2017} and \cite{Hoffmann2018a}. 
They effectively test whether $\nu_s([z,\infty))=\nu_0([z,\infty))$ for all $s$ and all $z>\delta$. 
As $\delta$ is fixed, this test does not necessarily detect a change of the jump activity index which is only characterized by the behavior of the jump measure near zero.
A modification presented by \cite{Hoffmann2018} is to introduce the functions $\varphi_s(z) = \int_{z}^\infty \varphi(y)\nu_s(dy)$ for a function $\varphi$ such that $\varphi_s(0)$ is finite. 
Then one can test whether $\varphi_s(z)=\varphi_0(z)$ for all $s$ and all $z$.
In principle, this approach can detect changes of the jump activity index. 
However, the effect of small jumps is dampened by the function $\varphi$, and thus the procedure might not be very sensitive to changes of $\nu_s$ around zero. 
Moreover, they require the characteristics of the process to vary smoothly in the sense that a sequence $X^n$ of stochastic processes is observed, with jump compensators $\nu^n_s = \tilde{\nu}_{\frac{s}{T}}$ for a deterministic measure-valued function $t\mapsto \tilde{\nu}_t.$
This approach is common for change-point detection, but distinguishes their framework from the statistical setting considered in the present paper.
A change-point test which is tailor-made for changes of the jump activity index $\alpha$ is presented by \cite{todorov2017testing}.
The model therein is a non-Markovian semimartingale which is a pure-jump process, and the limit distribution of the test statistic is derived only for the case of constant jump activity.

In a pure high-frequency setting of $n$ equidistant observations at frequency $1/n$, the optimal rate of convergence for estimation of a constant value $\alpha$ in the presence of a diffusion component is conjectured to be $n^{\alpha/4}$, up to logarithmic factors.
This rate has been derived by \cite{ait2012identifying} upon analyzing the Fisher information matrix. 
An early estimator due to \cite{ait2009estimating} achieved the rate $n^\frac{\alpha}{10}$, which has been improved to $n^{\frac{\alpha}{8}}$ by \cite{jing2012jump}.
For arbitrary $\epsilon>0$, there exist estimators which converge at rate $n^{\alpha/4-\epsilon}$, see \cite{reiss2013testing} for the L{\'e}vy case and \cite{bull2016near} for more general Itô semimartingales. 
In concurrent work, we show that the optimal rate $n^{\frac{\alpha}{4}}$ can be achieved for the Lévy case, up to logarithmic factors \citep{Mies2019}.
The situation is different if the diffusion component is absent, $\sigma\equiv 0$. 
In this pure-jump case, a rate of $\sqrt{n}$ can be attained \citep{todorov2015jump}. 
Except for the mentioned change-point test of \cite{todorov2017testing}, all available estimators assume $\alpha(x)\equiv \alpha$ to be constant and thus can not be applied directly to the present situation. 

Our proposed estimator for $\alpha(x)$ achieves a rate of convergence of at least $\sqrt{Tb} h^{-\alpha(x)/8}$, and may be improved by leveraging prior knowledge when choosing tuning parameters, see section \ref{sec:stable-like}. 
There is no matching benchmark for the present nonparametric situation in the literature. 
Potentially, the dependence on the bandwidth $b$ could be improved by stricter smoothness assumptions and a higher order kernel. 
When neglecting the term $\sqrt{b}$ due to smoothing and $\sqrt{T}$ due to the different sampling scheme, we find that the remaining factor $h_n^{-\alpha(x)/8}$ matches the rate of \cite{jing2012jump} for the case of constant $\alpha$ under high-frequency asymptotics. We thus conjecture that our estimator suffers from the same inefficiency, as the desired rate of convergence in the latter case is $h_n^{-\alpha/4}$. 
This is to be expected, since our estimator is essentially a localized version of the estimator of \cite{jing2012jump}.
Transferring other estimators for constant $\alpha$ to the state-dependent situation is thus of interest for future work.

\subsection*{Outline}

The remainder of this paper is structured as follows. In section \ref{sec:generator}, we present the methodology to bound the conditional moments of nonlinearly transformed increments of jump diffusion process by an analytical investigation of the Markov transition semigroup. The derived bounds on the infinitesimal generator might be of independent interest. These approximations are used to construct an asymptotically normal estimator of the drift in section \ref{sec:drift}. Inference for the small jumps, i.e.\ jump activity estimation, is treated in section \ref{sec:stable-like}. We demonstrate the applicability of the results by a simulation study in section \ref{sec:simulations}. All technical proofs are postponed to section \ref{sec:proofs}.

\subsection*{Notation}

To abbreviate some notation, we introduce $a\wedge b=\min(a,b)$ and $a\vee b=\max(a,b)$ for $a,b\in\R$. 
The space of $k$ times continuously differentiable functions $f:\R\to \R$ is denoted by $\mathcal{C}^k=\mathcal{C}^k(\R;\R)$. By $f'$, we denote the derivative of $f$, and higher order derivatives are written as $f''$ and so on, or as $f^{(k)}, k\in\N$. The uniform norm is $\|f\|_\infty = \sup_x |f(x)|$. Weighted uniform norms $\|f\|_{\infty, p}$ are defined in the sequel, as well as the shift operator $\tau_x$ and the linear operators $\mathcal{A}, \mathcal{A}^*, \mathcal{J},\mathcal{J}^*$ and the scaled fractional derivative $f^{[\alpha]}$. For random variables $A_n, B_n$, the expression $A_n=\mathcal{O}_P(B_n)$ means that $A_n/B_n$ is bounded in probability, and $A_n = o_P(B_n)$ means $A_n/B_n \pconv 0$, where $\pconv$ indicates convergence in probability. Weak convergence of distributions and random variables is denoted as $\wconv$.

\section{Conditional moments of jump diffusions}\label{sec:generator}

The estimators proposed in the sequel will be based on nonlinearly transformed increments of the semimartingale $X_t$ given by \eqref{eqn:SDEdef}. Central to our analysis are expressions for the corresponding conditional moments, i.e.\ quantities of the form $E(f(X_{t+h}-X_t)| X_t)$. This includes the rescaled increments $f(u(X_{t+h}-X_t))$ by setting $f_u(x) = f(ux)$. The expressions derived in this section allow us to build suitable nonparametric estimators in subsequent sections, as well as enabling the study of their asymptotic behavior. To this end, let $f\in\mathcal{C}^2$ be a twice continuously differentiable function. Furthermore, denote by 
\begin{align*}
	\mathcal{A} f (x) = \mu(x) f'(x) + \frac{\sigma^2(x)}{2} f''(x) + \int_{\R} \left[ f(x+z)-f(x) - f'(x)z\mathds{1}_{|z|\leq 1}\right]\rho(x,z)\, dz
\end{align*}
the formal infinitesimal generator of the Markov process $X_t$. Then It{\^o}'s formula yields for any $h>0$ and any $t\geq 0$, \begin{align*}
	f(X_{t+h})= f(X_t) + \int_0^h \mathcal{A}f(X_{t+s})\, ds + (M_{t+h}-M_t),
\end{align*}
where $M_t$ is a local martingale. If $f$ is a bounded function, and $\mathcal{A}f$ is bounded as well, then $M_t$ is in fact a martingale and integrating yields \begin{align}
	\nonumber T_hf(x) = E\left(f(X_{t+h})|X_t=x\right) &= f(x) + \int_0^h E\left(\mathcal{A}f(X_{t+s})|X_t=x\right)\,ds \\
		\label{eqn:ito-expansion}	&= f(x) + \int_0^h T_s \mathcal{A}f(x)\, ds.
\end{align}
It is convenient to formulate this identity in terms of the conditional expectation operator $T_h$, which maps continuous bounded functions $f$ to continuous bounded functions $T_hf$ as defined above. Due to the Markov property, the family $\{T_h, h>0\}$ constitutes a semigroup of contractions on this space, generated by the operator $\mathcal{A}$. For details about the semigroup approach to continuous-time Markov processes, we refer to \cite{ethier2009markov}. In our situation, this approach allows us to study the conditional moments in a purely analytical fashion. In particular, \eqref{eqn:ito-expansion} can be interpreted as a first-order expansion of the conditional expectation $T_hf(x)$. For our statistical applications, we will be able to choose $f$ sufficiently regular such that $\mathcal{A}^2f\in\mathcal{C}^2$ is bounded as well. We may then iterate the expansion \eqref{eqn:ito-expansion} to obtain \begin{align}
	\nonumber T_hf(x) &= f(x) + h \mathcal{A}f(x) + \int_0^h \int_0^s T_r \mathcal{A}^2 f(x)\, dr\, ds\\
	\label{eqn:taylor-1}&= f(x) + h \mathcal{A}f(x) + \mathcal{O}(h^2 \|\mathcal{A}^2f\|_\infty),
\end{align}
since the operator $T_r$ is a contraction w.r.t.\ the uniform norm. In the sequel, we will be interested in conditional expectations of increments of the form $f(X_{t+h}-X_t)$. This can be cast into the previous framework upon noting that $E(f(X_{t+h}-X_t)|X_t=x) = E(f(X_{t+h}-x)|X_t=x)$, which can equivalently be written as $T_h \tau_xf(x)$, where $\tau$ is the shift operator $\tau_xf(y)=f(y-x)$. The leading order term of this conditional expectation is \begin{align}
	\label{eqn:Astardef}\begin{split}\mathcal{A}^*f(x) = \mathcal{A}\tau_xf(x)
	&= \mu(x)f'(0) + \frac{\sigma^2(x)}{2} f''(0) \\
	&\qquad + \int_{\R} \left[ f(z)-f(0) -f'(0)z \mathds{1}_{|z|\leq 1}\right]\rho(x,z)\, dz.
	\end{split}
\end{align}
An appealing property of expression \eqref{eqn:Astardef} is that it depends on $x$ only via the spot characteristics of the process $X_t$, i.e.\ the quantities $\mu(x), \sigma(x)$ as well as the local jump measure.

The statistical benefit of this analytical approach, which we will leverage in the subsequent sections, is that it allows us to approximate conditional expectations and variances non-asymptotically at a known rate. We summarize this as follows.

\begin{proposition}\label{prop:cond-approx}
For any function $f \in \mathcal{C}^2(\R)$ such that $\mathcal{A}\tau_xf\in\mathcal{C}^2$ and $f(0)=0$,
	\begin{align*}
		\left|E(f(X_{t+h}-X_t)|X_t=x) - h \mathcal{A}^*f(x) \right| &\leq h^2 \|\mathcal{A}^2 \tau_xf\|_\infty	\\
		\left|\Var(f(X_{t+h}-X_t)|X_t=x) - h \mathcal{A}^* f^2(x) \right| &\leq h^2 \|\mathcal{A}^2\tau_xf^2\|_\infty + 2 h^2 |\mathcal{A}^*f(x)|^2\\
		&\qquad + 2h^4 \|\mathcal{A}^2\tau_xf\|_\infty^2.
	\end{align*}
\end{proposition}

To make the approximation of Proposition \ref{prop:cond-approx} effective, we are looking for conditions to ensure $\|\mathcal{A}^2 f\|_\infty < \infty$. 
It turns out that assumptions (J1) and (J2) are suitable for this purpose.

To conveniently formulate our results, we introduce a family of weighted supremum norms for a function $f:\R\to\R$, given by \begin{align*}
	\|f\|_{\infty, p} = \sup_x |f(x)| (|x|\vee 1)^p, \quad p\in \R.
\end{align*}
These norms are ordered such that $\|f\|_{\infty, q} \leq \|f\|_{\infty,p}$ whenever $q\leq p$. For $p=0$, we recover the usual uniform norm, and we will use the notation $\|f\|_{\infty, 0}=\|f\|_{\infty}$ interchangeably. Furthermore, the weighted norms satisfy a duality of the form \begin{align}
	\|f \cdot g\|_{\infty, p} \leq \|f\|_{\infty, p+q} \|g\|_{\infty, -q} \label{eqn:norm-duality}
\end{align}
for any two functions $f,g:\R\to\R$ and any $p,q\in\R$.
Note that an equivalent family of norms is used by \cite[Proposition 2]{Amorino2019}, for $p\leq -1$.

To treat the full generator $\mathcal{A}f$, we impose growth conditions on the drift and diffusion functions $\mu$ resp.\ $\sigma$.

\begin{itemize}
	\item[\textbf{(D)}] $\mu\in\mathcal{C}^2$, $\|\mu\|_{\infty, -1},\|\mu^{(k)}\|_{\infty, -p_D} \leq C_\mu < \infty$, $k=1,2$, for some $p_D\geq 0$. 
	\item[\textbf{(V)}] $\sigma^2\in \mathcal{C}^2$, $\|\sigma^2\|_{\infty, 0}, \|(\sigma^2)^{(k)}\|_{\infty, -p_V}\leq C_\sigma < \infty$, $k=1,2$, for some $p_V\geq 0$. 
\end{itemize}

For our statistical applications, the values of $p_D, p_V$ can be quite large as they can be compensated by choosing very rapidly decaying functions $f$, as the following Theorem \ref{thm:A2full} shows. However, we need to impose stricter zeroth order conditions on $\mu$ and $\sigma^2$ since the admissible ranges of $q$ in the technical Lemma \ref{lem:C2domain} in the appendix are restricted.

\begin{theorem}\label{thm:A2full}
	If (J1) and (J2) as well as conditions (D) and (V) hold, then there exists a constant $K=K(\alpha,\tau,C_\rho, C_\mu,C_\sigma)>0$ such that for all $f\in\mathcal{C}^4$,  
	\begin{align*}
		\|\mathcal{A}^2f\|_{\infty, 0} \leq K\left( \|f'\|_{\infty, (\tau \vee p_V \vee p_D)+1} + \|f''\|_{\infty, (\tau \vee p_V \vee p_D)+1} + \|f'''\|_{\infty, \tau \vee p_V \vee p_D} + \|f''''\|_{\infty, 0} \right).
	\end{align*}
\end{theorem}

The bound of Theorem \ref{thm:A2full} is in particular finite if $f$ is compactly supported or rapidly decaying in the Schwartz sense. Furthermore, the bound can be applied to bound $\|\mathcal{A}^2\tau_xf\|_\infty$ in Proposition \ref{prop:cond-approx} since $\|\tau_xf\|_{\infty, p} \leq 2^p(|x|\vee 1)^p \|f\|_{\infty, p}$, see Lemma \ref{lem:normsup} in the appendix. 

The formulation in terms of the weighted uniform norm, which might seem rather technical, arises because we allow the derivatives of $\mu$ and $\sigma$ to be unbounded, and $\mu$ to grow linearly. The latter property also prevents us from studying bounds of higher order, i.e.\ $\|\mathcal{A}^k f\|_{\infty, 0}$ for $k>2$. If we are willing to impose stricter boundedness conditions, we obtain the following stronger result.

\begin{theorem}\label{thm:Ak}
	Let (J1) and (J2) hold, and assume additionally that $|\frac{d^k}{dx^k} \rho(x,z)| \leq C_\rho (|z|^{-1-\alpha} \vee |z|^{-1-\tau})$ and $\|\mu^{(k)}\|_{\infty, 0}\leq C_\mu, \|(\sigma^2)^{(k)}\|_{\infty, 0}\leq C_\sigma$, for all $k=0,\ldots, 2m$. Then there exists a $K>0$ such that for all $f\in\mathcal{C}^{2k}$ \begin{align*}
		\|\mathcal{A}^k f\|_{\infty, 0} \leq K \sum_{j=1}^{2k} \|f^{(j)}\|_{\infty, 0},\qquad k=1,\ldots, m+1.
	\end{align*}
\end{theorem}

In the statistical setting of this paper, the boundedness assumptions of Theorem \ref{thm:Ak} are too strict. In particular, we will need the Markov process $X_t$ to be ergodic, which is typically ensured by a mean-reverting drift term, e.g.\ $\mu(x)=-x$ (see \cite[Lemma 2.4]{masuda2007ergodicity}). 
This choice of $\mu$ satisfies condition (D), but it precludes the application of Theorem \ref{thm:Ak}.
In the following, we will only make use of Theorem \ref{thm:A2full}, but the stronger Theorem \ref{thm:Ak} might be of independent interest.

\section{Nonparametric drift estimation}\label{sec:drift}

As we have seen in the previous section, by transforming the increments $X_{t+h}-X_t$ of the process nonlinearly and studying $f(X_{t+h}-X_t)$, conditional variances exist and can be handled by means of Proposition \ref{prop:cond-approx}. This holds true in spite of a potentially heavy-tailed jump process. As performed in a previous study for a more restricted model \citep{mies2018tempering}, we may use this property to construct an asymptotically normal estimator of the drift function $\mu(x)$. To this end, choose a sufficiently smooth, odd function $f$ with $f(0)=0$. By assuming the jumps to be symmetric, we obtain \begin{align*}
	\int_{\R} \left[ f(z)-f(0)-f'(0)z\mathds{1}_{|z|\leq 1} \right]\, \rho(x,z)\, dz = 0.
\end{align*}
Furthermore, $f''(0)=0$ since $f$ is odd.
Then Proposition \ref{prop:cond-approx} yields $E(f(X_{t+h}-X_t)|X_t=x) \approx h\mathcal{A}^*f(x)  = h\mu(x)f'(0)$. The drift can thus be identified by estimating this conditional expectation. 

If we have at hand an equidistant sample $X_{t_i}$, $t_i=ih, i=1,\ldots, n+1$, of the process $X_t$ solving \eqref{eqn:SDEdef}, we can estimate the conditional expectation by means of a nonparametric kernel smoother. 
To keep the exposition simple, we study the Nadaraya-Watson estimator of the following form
\begin{align*}
	\hat{\mu}_n(x) &= \frac{1}{\hat{m}_n(x)}\frac{1}{n} \sum_{i=1}^n \frac{f(X_{t_{i+1}} - X_{t_i})}{hf'(0)} G_b(X_{t_i}-x), \\
	\hat{m}_n(x) &= \frac{1}{n} \sum_{i=1}^n G_b(X_{t_i}-x).
\end{align*}
Here, $G$ is a smoothing kernel, $b>0$ is a bandwidth parameter and $G_b(y)=b^{-1} G(y/b)$. 
The function $f$ is a design parameter, and we assume the following regularity. 

\begin{itemize}
\item[\textbf{(F')}] The function $f:\R\to\R$ is bounded, odd, $f\in\mathcal{C}^4$, $f(0)=0$, $f'(0)\neq 0$, and the derivatives $f^{(k)}$ decay faster than polynomially for $k=1,2,3,4$, i.e.\ $\|f^{(k)}\|_{\infty, p}<\infty$ for all $p>0$.
\end{itemize}

By choosing $f$ sufficiently regular and bounded, we are also able to handle conditional variances via Proposition \ref{prop:cond-approx}.
In particular, the variance of $\hat{\mu}_n(x)$ is finite although the increment $X_{t_{i+1}}-X_{t_i}$ might be heavy-tailed.

Furthermore, we need some technical assumptions on the kernel function $G$ and the rate $b$, as well as ergodicity of the process $X_t$.

\begin{itemize}
	\item[\textbf{(K1)}] $G:\R\to [0,\infty)$ is bounded and compactly supported, $G(y)=0$ for $|y|>1$, and $\int_{-1}^1 G(y)dy =1$.
	\item[\textbf{(K2)}] $X_t$ is stationary with density $X_t\sim m$, and geometrically $\alpha$-mixing. The initial value $X_0\sim m$ is independent of the Brownian motion $B$ and the Poisson counting measure $N$.
\end{itemize}

The geometric ergodicity of (K2) is required to ensure consistency of the kernel density estimator $\hat{m}_n(x)\to m(x)$. For sufficient conditions in terms of $\mu,\sigma$ and $c$, and for a definition of the $\alpha$-mixing coefficients, we refer to \cite{masuda2007ergodicity}. 
See also Section \ref{sec:ergodicity}.

\begin{theorem}\label{thm:drift-clt}
Let (J1), (J2), (D), (V), (K1) and (K2) hold, and $f$ satisfy (F'). If $b\to 0, h\to 0$ and $Tb\to\infty$ as $n\to\infty$, then \begin{align*}
	|\hat{\mu}_n(x)-\mu(x)| = \mathcal{O}_P\left(\frac{1}{\sqrt{Tb}}\vee b \vee h\right).
\end{align*} 
If furthermore $Tb^3\to 0, Tbh^2\to 0$, then \begin{align*}
	\sqrt{Tb} (\hat{\mu}_n(x)-\mu(x)) \wconv \mathcal{N}\left( 0, \frac{\mathcal{A}^*f^2(x)}{f'(0)^2} \frac{\int G^2(y)dy}{m(x)} \right).
\end{align*}
\end{theorem}

Note that for our choice of $f$, i.e.\ odd such that $f(0)=0$, \begin{align*}
	\mathcal{A}^*f^2(x) &= \frac{\sigma^2(x)}{2} f'(0)^2 + \int f^2(z)\rho(x,z)\,dz \\
	&= \frac{\sigma^2(x)}{2} f'(0)^2 + \mathcal{J}^*f^2(x).
\end{align*}
The first term is due to the diffusion component, while the second term represents the variance due to the jump component. 
It is possible to eliminate the effect of the jumps asymptotically.
To this end, introduce an additional scaling factor $u$ and consider the drift estimator based on $f_u(\cdot) = f(u\,\cdot)$. 
By Theorem \ref{thm:drift-clt}, the asymptotic variance of $\hat{\mu}_n(x)$ is then determined by 
\begin{align*}
	\mathcal{A}^*f_u^2(x)/f_u'(0)^2
	&= (\sigma^2(x) (f^2)''(0)/2 + \mathcal{J}^*f^2_u(0)/u^2)/f'(0)^2 \\
	&= \sigma^2(x) + \frac{\mathcal{J}^*f^2_u(0)}{u^2f'(0)^2}.
\end{align*}
It can be shown that our assumptions on $\rho$ imply that $\mathcal{J}^*f^2_u(0)=\mathcal{O}(u^\alpha)$, such that the contribution of the jumps vanishes if $u\to\infty$. 
This leads to a reduction of the asymptotic variance. 
As Theorem \ref{thm:drift-clt} does not allow for a sequence of functions $f_{u_n}$, we slightly adopt its proof to obtain the following result.

\begin{theorem}\label{thm:drift-clt-filtered}
	Suppose that the conditions of Theorem \ref{thm:drift-clt} hold with $Tb^3\to\infty, Tbh^2\to 0$. 
	For a sequence $u=u_n\to\infty$ such that $Tbh^2u^8\to 0$, denote 
	\begin{align*}
		\hat{\mu}^*_n(x) &= \frac{1}{\hat{m}_n(x)}\frac{1}{n} \sum_{i=1}^n \frac{f_{u_n}(X_{t_{i+1}} - X_{t_i})}{hu_nf'(0)} G_b(X_{t_i}-x).
	\end{align*}
	Then 
	\begin{align*}
		\sqrt{Tb} (\hat{\mu}^*_n(x)-\mu(x)) \wconv \mathcal{N}\left( 0, \sigma^2(x) \frac{\int G^2(y)dy}{m(x)} \right).
	\end{align*}
\end{theorem}

Note that the sequence $u_n$ may tend to infinity arbitrarily slowly, such that it can always be chosen in accordance with Theorem \ref{thm:drift-clt-filtered}, e.g.\ as $u=(Tbh^2)^{1/16}$.
Hence, we find that the rate of convergence and the asymptotic variance of the estimator $\hat{\mu}_n^*(x)$ is the same as for the Gaussian case without jumps, see \cite[Thm.\ 3]{bandi2003fully}.
The same holds true for the thresholded drift estimator of \cite{Mancini2011}, which however requires the jumps to be of finite activity.

The introduction of a scaling sequence $u_n\to\infty$ to reduce the asymptotic variance can be regarded as a smooth analogue of filtering i.e.\ eliminating the jumps. 
This jump-filtering has recently been applied for parametric estimation of the drift functional by \cite{gloter2018jump}, leading to an efficient estimator.
They consider the jump term as a nuisance, and assume $\sigma(x)$ to be known. 
Therein, the authors introduce a truncation sequence $v_n\gg \sqrt{h}$ and consider the truncated increments $(X_{t+h}-X_t)\mathds{1}_{|X_{t+h}-X_t|\leq v_n}$.
To achieve efficiency in this setting, the authors introduce a truncation sequence $v_n\gg \sqrt{h}$ and consider the truncated increments $(X_{t+h}-X_t)\mathds{1}_{|X_{t+h}-X_t|\leq v_n}$.
If we denote $g(y)=y\mathds{1}_{|y|\leq 1}$, the latter truncated increment corresponds to $g_{u_n}(X_{t+h}-X_t)$ with $u_n=1/v_n$.
Note that this truncation is not covered by our analysis since we require $g$ to be smooth.
In contrast to our model, \cite{gloter2018jump} require the jumps to be of finite variation, while allowing them to be non-symmetric. 

A central contribution of \cite{gloter2018jump} is to relax the assumptions on the sampling scheme.
In particular, they require that $nh^p\to 0$ for some $p>1$, and larger values of $p$ represent less restrictive conditions.
The value of $p$ depends on the jump measure, and is at best $p=3-\epsilon$ for some small $\epsilon$ in the case where the jump density $\rho(x,z)$ is bounded in $z$.
If the jumps are of infinite activity, $p$ is smaller.
For our filtered estimator $\hat{\mu}_n^*(x)$, we require $Tbh^2\to 0$.
Neglecting the nonparametric bandwidth, this roughly corresponds to $Th^2=nh^3\to 0$ and is thus less restrictive than the conditions of \cite{gloter2018jump}.
This improvement is only possible because we restrict the jump measure to be symmetric. 

Another advantage of the jump filtered estimator $\hat{\mu}_n^*(x)$ over $\hat{\mu}_n(x)$ is that the asymptotic variance of the former has a much simpler form.
To use Theorem \ref{thm:drift-clt-filtered} in practice, we only need a consistent estimator of the asymptotic variance, i.e., of the spot volatility $\sigma(x)$.
We propose the estimator 
\begin{align}
	\hat{\sigma}_n^2(x) = \frac{1}{\hat{m}_n(x)}\frac{1}{n} \sum_{i=1}^n \frac{f_{u_n}(X_{t_{i+1}} - X_{t_i})}{hu_n^2 f''(0)/2} G_b(X_{t_i}-x),
\end{align}
for a function $f$ which satisfies condition (F'') below.

\begin{itemize}
\item[\textbf{(F'')}] The function $f:\R\to\R$ is bounded, $f\in\mathcal{C}^4$, $f(0)=0$, $f'(0)= 0$, $f''(0)\neq 0$, and the derivatives $f^{(k)}$ decay faster than polynomially for $k=1,2,3,4$, i.e.\ $\|f^{(k)}\|_{\infty, p}<\infty$ for all $p>0$.
\end{itemize}

\begin{theorem}\label{thm:vola-consistent}
	Let (J1), (J2), (D), (V), (K1) and (K2) hold, and $f$ satisfy (F'). If $b\to 0, h\to 0$, $Tb\to\infty$, and $u=u_n\to\infty$ as $n\to\infty$, then
	\begin{align*}
		\hat{\sigma}^2_n(x) \pconv \sigma^2(x).
	\end{align*}
\end{theorem}

We note that a central limit theorem for $\hat{\sigma}_n^2$ in the present situation, although desirable, is hard to obtain because we allow for jumps of infinite variation. 
For the related problem of estimating the integrated volatility of a semimartingale, the activity of the jumps determines the attainable rate of convergence \citep{jacod2014remark}, and the estimators are typically dominated by a bias term in the infinite variation case.
To obtain a fast rate of convergence and a corresponding central limit theorem, additional assumptions may be imposed, see \cite{jacod2014efficient}.
Although the additional assumptions are very similar to the ones imposed in section \ref{sec:stable-like} below, a full nonparametric treatment of the volatility $\sigma(x)$ in a Markovian setting is out of scope of the present paper.

\section{Stable-like processes}\label{sec:stable-like}

We now proceed to apply the results of section \ref{sec:generator} to a more specialized class of processes. So far, assumption (J2) requires the jump activity index to be smaller than $\alpha<2$. We are interested in processes whose small jumps behave in a certain sense almost identical to those of an $\alpha(x)$-stable L{\'e}vy process. This stable-like behavior is formalized as follows.

\begin{itemize}
	\item[\textbf{(SL1)}] The spot jump intensity measure admits a density $\rho(x,z)$ of the form \begin{align*}
		\rho(x,z)&=\frac{r(x)}{|z|^{1+\alpha(x)}}(1+g(x,z)),\quad |z|\leq 1,
	\end{align*} 
	for a function $g\geq -1$ which is symmetric in $z$, differentiable in $x$, and $|\frac{d^k}{dx^k}g(x,z)| \leq C_g|z|^{\delta(x)}$, $C_g>0$, $0<\delta(x)\leq \alpha(x) <2$, $k=0,1,2$. For $|z|>1$, $|\frac{d^k}{dx^k}\rho(x,z)|\leq C_\rho |z|^{-1-\tau}$, $\tau>1$.
	
	\item[\textbf{(SL2)}] The functions $x\mapsto r(x)$ and $x\mapsto \alpha(x)$ are $\mathcal{C}^2$ and bounded such that $\|r\|_{\infty}$, $\|r'\|_\infty$, $\|r''\|_\infty< \infty$ and $\|\alpha\|_\infty<2, \|\alpha'\|_\infty, \|\alpha''\|_\infty <\infty$. The function $x\mapsto \delta(x)>0$ is continuously differentiable on $\R$.
\end{itemize}

When referring to these assumptions, we will speak of a stable-like process. It can be checked that (J1),(SL1) and (SL2) together imply condition (J2) (see Lemma \ref{lem:SL-T-J}), such that the results of section \ref{sec:generator} apply.
Assumption (SL1) is very similar to the locally stable pure-jump processes studied by \cite{masuda2016non}, where $\alpha(x)$ is constant. These Lévy driven SDEs are contained in our framework, as the following argument shows.
An example with more complex jump dynamics is presented in section \ref{sec:simulations} below.

\begin{example}
	Let $Z_t$ be a symmetric pure-jump L{\'e}vy process with intensity measure $\rho(z)dz$, the density of which is of the form $\rho(z)=(1+g(z))|z|^{-1-\alpha}$ for some $\alpha\in(0,2)$, and $g(x,z)=g(z)$ satisfying (SL1) with $\delta(x)=\delta, r(x)=1$ constant. If we consider the Markov process $dX_t = l(X_{t-}) dZ_t$ for a smooth function $l$ such that $\|l\|_\infty, \|l'\|_\infty, \|l''\|_\infty < \infty$ and $l(x)\geq l_0>0$, the corresponding jump intensity measure has the density $\rho(x,z)=\rho(z/l(x))/l(x)$, which can be written as $\rho(x,z)=l(x)^\alpha |z|^{-1-\alpha}(1+g(z/l(x)))$. Since $l(x)$ is bounded from below, the function $x\mapsto l(x)^\alpha$ is as smooth as required by (SL2). To ensure the smoothness of $g(x,z)=g(z/l(x))$ as in (SL1), we need an additional requirement. A sufficient condition is $|g(z)| + |zg'(z)| + |z^2g''(z)| \leq C_g |z|^{\delta}$ for $|z|\leq 1$ and $|g(z)| + |zg'(z)| + |z^2g''(z)| \leq C_g |z|^{\alpha-\tau}$ for $|z|> 1$.
\end{example}

The similarity of a stable-like process to an $\alpha$-stable L{\'e}vy process is underscored by an investigation of the jump part of the generator term $\mathcal{A}^*f(x)$, which is given by \begin{align*}
	\mathcal{J}^*f(x) = \int_{\R} \left[f(z)-f(0)-f'(0)z\mathds{1}_{|z|\leq 1}\right] \rho(x,z)dz.
\end{align*}
If $\rho(x,z)=|z|^{-1-\alpha(x)}$, we define the expression \begin{align}
	\label{eqn:fracderiv}f^{[\alpha(x)]}(y) = \int_{\R} \frac{f(y+z)-f(y)-f'(y)z\mathds{1}_{|z|\leq 1}}{|z|^{1+\alpha(x)}}dz = \frac{1}{B_{\alpha(x)}} f^{(\alpha(x))}(y).
\end{align}
The term $f^{[\alpha(x)]}(y)$ is finite for $\alpha(x)\in (0,2)$, since a Taylor expansion of the integrand yields $|f^{[\alpha(x)]}(y)| \leq 2\|f\|_{\infty}/\alpha(x) + 2\|f''\|_\infty/(2-\alpha(x))$ for any $f\in \mathcal{C}^2$.

For rescaled functions $f_u(x)=f(ux)$, formula \eqref{eqn:fracderiv} yields the scaling behavior $f^{[\alpha]}_u(0)=u^\alpha f^{[\alpha]}(0)$. As $u\to\infty$, the same behavior holds true asymptotically for $\mathcal{J}^*f_u(x)$ for stable-like processes.
 
\begin{lemma}\label{lem:locstable-generator}
	If (J1) and (SL1) hold, then for $u\geq 1$, \begin{align*}
		\left|\mathcal{J}^*f_u(x) -  u^{\alpha(x)} r(x) f^{[\alpha(x)]}(0)\right|\leq K_x \tilde{C}_\rho  \left( \|f''\|_{\infty, 0} + \|f\|_{\infty, 0} \right)u^{\alpha(x)-\delta(x)},
	\end{align*}
	where $K_x = \frac{(r(x)+1)}{(2-\alpha(x))(\alpha(x)-\delta(x))}$.
\end{lemma}

Together with Proposition \ref{prop:cond-approx} and Theorem \ref{thm:A2full}, we obtain an approximation of the conditional expectation and variance of the increments $f_u(X_{t+h}-X_t)$, for the following class of functions.

\begin{itemize}
	\item[\textbf{(F)}] The function $f:\R\to\R$ is bounded, $f\in\mathcal{C}^4$, $0=f(0)=f'(0)=f''(0)$ and the derivatives $f^{(k)}$ decay faster than polynomially for $k=1,2,3,4$.
\end{itemize}

\begin{theorem}\label{thm:cond-approx-SL}
Let (J1), (D), (V), (SL1) and (SL2) hold, and let $u\geq 1$. There exists a constant $\tilde{K}_x$ and for any function $f$ satisfying (F), a constant $\tilde{K}_f<\infty$, such that
	\begin{align*}
		\left|E(f_u(X_{t+h}-X_t)|X_t=x) - h u^{\alpha(x)}r(x)f^{[\alpha(x)]}(0) \right| &\leq \tilde{K}_f\tilde{K}_x  \left( h^2u^4  + hu^{\alpha(x)-\delta(x)} \right),\\
		\left|\Var(f_u(X_{t+h}-X_t)|X_t=x) - h u^{\alpha(x)} r(x) (f^2)^{[\alpha(x)]}(0) \right| &\leq \tilde{K}_f\tilde{K}_x  \left( h^2u^4 + hu^{\alpha(x)-\delta(x)} +  h^4u^8 \right).
	\end{align*}
	The constant $\tilde{K}_x$ satisfies $\tilde{K}_x= \frac{(|x|\vee 1)^{2q}}{\alpha(x)^2(\alpha(x)-\delta(x))^2}$ for $q=(\tau\vee p_D \vee p_V)+1$.
\end{theorem}

We will now construct an estimator for the state-dependent jump activity index $\alpha(x)$ of a stable-like process. From Theorem \ref{thm:cond-approx-SL}, we know that $\alpha(x)$ can be identified via the scaling behavior of nonlinearly transformed increments $f_u(X_{t+h}-X_t)$. As a first step, we study the following statistic, which is inspired by the Nadaraya-Watson estimator, 

\begin{align*}
	\hat{R}_n(x) &= \frac{1}{\hat{m}_n(x)}\frac{1}{n} \sum_{i=1}^n \frac{f(u_n(X_{t_{i+1}} - X_{t_i}))}{hu^{\alpha(x)}} G_b(X_{t_i}-x) \\
	\hat{m}_n(x) &= \frac{1}{n} \sum_{i=1}^n G_b(X_{t_i}-x).
\end{align*}

It turns out that $\hat{R}_n(x)$ converges to $r(x)f^{[\alpha(x)]}(0)$ in probability. Now we construct a second statistic based on the function $f_\gamma(y)=f(\gamma y)$ for some $\gamma\neq 1$, given by 
\begin{align*}
	\hat{R}_n(x,\gamma) &= \frac{1}{\hat{m}_n(x)}\frac{1}{n} \sum_{i=1}^n \frac{f_\gamma(u_n(X_{t_{i+1}} - X_{t_i}))}{hu^{\alpha(x)}} G_b(X_{t_i}-x).
\end{align*}
Then $\hat{R}_n(x,\gamma)$ will converge to $r(x)f_\gamma^{[\alpha(x)]}(0) = \gamma^{\alpha(x)}r(x)f^{[\alpha(x)]}(0)$. The ratio of these two statistics is thus approximately $\gamma^{\alpha(x)}$. 
Note also that by taking the ratio $\hat{R}_n(x,\gamma)/\hat{R}_n(x)$, the unknown factor $u^{\alpha(x)}$ cancels, making the ratio a feasible estimator.
Taking the logarithm, we estimate $\alpha(x)$ via 

\begin{align*}
	\hat{\alpha}_n(x) =  \hat{\alpha}_n(x,\gamma) &= -\frac{1}{\log{\gamma}}\log \frac{\hat{R}_n(x)}{\hat{R}_n(x,\gamma)}.
\end{align*}

By using this estimator as a plug-in, we may also consistently estimate $r(x)$ via 
\begin{align*}
	\hat{R}_n^*(x) = \hat{R}_n^*(x,\gamma) = \frac{1}{n f^{[\hat{\alpha}_n(x,\gamma)]}(0)} \sum_{i=1}^n \frac{f(u_n(X_{t_{i+1}}-X_{t_i}))}{hu^{\hat{\alpha}_n(x,\gamma)}} G_b(X_{t_i}-x).
\end{align*}

If we impose the following conditions on $b$, $u$ and the sampling scheme, both estimators $\hat{\alpha}_n(x)$ and $\hat{R}^*_n(x)$ are consistent and asymptotically normal.

\begin{itemize}
	\item[\textbf{(K3)}] $h=h_n\to 0$, $u=u_n\to \infty$ and $b=b_n\to 0$ such that $Tb \to \infty$, and for a sequence $\psi_n\leq\sqrt{Tbu^{\alpha(x)}}$, $\psi_n\to\infty$, \begin{align}
		\label{eqn:rates-Tb}\psi_nhu^{4-\alpha(x)}\to 0,\qquad \psi_n u^{-\delta(x)} \to 0 \qquad \psi_nb \log u \to 0 .
	\end{align}
\end{itemize}

The rate constraints (K3) will yield consistency with rate $\psi_n$ for $\psi_n\leq\sqrt{Tbu^{\alpha(x)}}$, and asymptotic normality if (K3) holds for $\psi_n=\sqrt{Tbu^{\alpha(x)}}$. In the latter case, \eqref{eqn:rates-Tb} reads as \begin{align}
	\label{eqn:rate-clt}Tbh^2u^{8-\alpha(x)}\to 0,\qquad Tb u^{\alpha(x)-2\delta(x)} \to 0 \qquad Tb^3 u^{\alpha(x)} (\log u)^2 \to 0.
\end{align} 
The conditions \eqref{eqn:rates-Tb} control, respectively, the bias due to approximating the conditional expectation by the generator, the bias incurred by approximating the generator only via the small jumps, and the continuity bias due to the kernel smoothing. 
In the simpler case where $X_t$ is a Lévy process and $\alpha$ is constant, only the second bias term would be present.
For the stronger conditions \eqref{eqn:rate-clt} to be feasible, we need $\delta(x)>\frac{\alpha}{2}$. 
Otherwise, the bias of neglecting the lower order behavior of small jumps will dominate the error of estimation. 
This requirement is in line with results on the asymptotic normality for estimators of a constant jump activity index, e.g.\ \cite[Assumption 1(v)]{bull2016near}.

\begin{theorem}\label{thm:alpha-clt}
	Let (J1), (D), (V), (SL1), (SL2) and (K1)-(K3) hold, with $\psi_n=\sqrt{Tbu^{\alpha(x)}}$. Then for any $f$ satisfying (F), and any $\gamma\neq 1$,
	\begin{align*}
			\sqrt{Tb u^{\alpha(x)}} \begin{pmatrix}
			 \hat{\alpha}_n(x,\gamma) - \alpha(x,\gamma)\\ 
			 \frac{\hat{R}_n^*(x,\gamma) -r(x)}{r(x)\log u} 
			\end{pmatrix} \wconv  W \cdot \begin{pmatrix}	1 \\ 1\end{pmatrix},
	\end{align*}
	for a Gaussian random variable $W\sim \mathcal{N}(0, s^2_\gamma(x))$, with asymptotic variance 
	\begin{align*}
			s^2_\gamma(x) = \frac{\int G^2(y)dy}{r(x)m(x) f^{[\alpha(x)]}(0)^2(\log\gamma)^2} \int \frac{(f(z)-\gamma^{-\alpha(x)} f(\gamma z))^2}{|z|^{1+\alpha(x)}}dz.
	\end{align*}
	If instead (K3) holds with $\psi_n=o(\sqrt{Tbu^{\alpha(x)}})$, then $\psi_n|\hat{\alpha}_n(x,\gamma)-\alpha(x)|\pconv 0$. If in the latter case $\psi_n\geq \log u_n$, then $\frac{\psi_n}{\log u_n}|\hat{R}_n^*(x,\gamma)-r(x)| \pconv 0$.
\end{theorem}

Note that the asymptotic covariance of Theorem \ref{thm:alpha-clt} is degenerate, as the error due to estimation of $\alpha(x)$ is dominant.
Furthermore, the asymptotic variance depends on the underlying process only via the values $\alpha(x)$ and $r(x)$, such that a plug-in estimator allows for feasible inference.

The asymptotic variance $s_\gamma(x)^2$, up to the factor $m(x)$ due to local smoothing, is identical to the asymptotic variance obtained by \cite{jing2012jump} where $X$ is a general semimartingale and the jump activity index is constant, see also \cite[Ch.\ 11]{ait2014high}. For a fast rate of convergence, it is crucial to choose $u$ as large as possible. In light of the constraint $Tbh^2u^{8-\alpha(x)}\to 0$, we need at least $u_n=o(h^{-\frac{2}{8-\alpha}})$. One option is to choose $u_n=h_n^{-1/4}$, and $b_n$ sufficiently small, leading to a rate of convergence of $\sqrt{Tb_n}h_n^{-\alpha(x)/8}$. 
This situation is analogous to the estimation of a constant jump activity index, where established estimators converge at rate $u_n^{\frac{\alpha}{2}}$ for an analogous scaling sequence $u_n$, see \cite{ait2009estimating,jing2012jump,bull2016near}.
Note that \cite{jing2012jump} require $u_n = o(h^{-\frac{1}{2+\alpha}})$, which is stricter than our requirement if we let $\alpha$ vary in $(1,2)$. Faster rates can be achieved by considering multiple time scales, as done by \cite{bull2016near}, to allow for $u_n=o(h^{-\frac{1}{2}-\epsilon})$ for arbitrary $\epsilon>0$. We choose not to pursue the multi-scale methodology in this paper.
Of course, our rates are not directly comparable to the previous ones, since we are working in a sampling situation where $T\to\infty$, and we need to average locally since $\alpha(x)$ is assumed to be non-constant. Thus, the rate of convergence is additionally affected by the smoothness of the signal, which could in principle be exploited by higher order kernels.

When applying these estimators to finite samples, the choice of the right scaling factor $u$ is rather intricate. To demonstrate this, neglect the factors $T$ and $b$ for a moment. If $u$ is too large, the approximation of the conditional expectation by means of the generator leads to a bias term of order $hu^{4-\alpha(x)}$, and the proportionality factor might be large as it depends on the first four derivatives of $f$. On the other hand, choosing $u$ too small increases the sampling noise, which is of order $u^{-\alpha(x)/2}$. This in turn incurs an additional bias term of order $u^{-\alpha(x)}$, which occurs due to the application of the delta method applied to the logarithm. Thus, the wrong choice of the scaling factor $u$ may lead to a bias in both cases, if $u$ is too large or too small. This is different from e.g.\ the choice of a bandwidth in standard nonparametric regression, where a smaller bandwidth only increases the variance of the estimator. We note that this phenomenon also occurs for other estimators of the jump activity index which are based on the same approach.

\section{Simulation study}\label{sec:simulations}

As an example, we study the model \begin{align*}
	dX_t &= -X_t dt + dB_t + dJ_t + dL_t,
\end{align*}
where $B_t$ is a standard Brownian motion and $L_t$ is a compound Poisson process with intensity $1$ and jump sizes according to a Student's t distribution with parameter $\tau=1.2$ and density which scales like $|x|^{-1-\tau}$ as $|x|\to\infty$. 
The jump process $J_t$ is of the form \eqref{eqn:SDEdef}, with spot jump intensity $\nu_s(dz)=|z|^{-1-\alpha(X_{s-})}dz$ for $\alpha(x)=1.9-\frac{\arctan(x)^2}{\pi^2}\in[1.65,1.9]$. 
Indeed, such a process can be constructed by choosing $\nu(dz)=|z|^{-\alpha(0)}$ and \begin{align*}
	c(x,z)=\text{sign}(z) |z|^\frac{\alpha(0)}{\alpha(x)} \left( \frac{\alpha(0)}{\alpha(x)} \right)^\frac{1}{\alpha(x)}.
\end{align*}
In particular, the Lipschitz continuity of $\alpha(x)$ implies $|c(x,z)-c(y,z)| \leq C|x-y| |z|\log|z|$ for $|z|\leq 1$. Thus, existence of $X_t$ for any initial condition can be ensured by \cite[Thm.\ 6.2.9]{applebaum2009levy}. 
The relevant local quantities for our example are \begin{align*}
	\mu(x)&=-x,&\qquad \sigma(x)=1,\\
	\alpha(x)&=1.9-\frac{\arctan(x)^2}{\pi^2},&\qquad r(x)=1.
\end{align*}

The paths of this process are simulated by means of an Euler scheme.
As a sampling scheme, we choose the time horizon $T=10$ and a sampling frequency of $h=10^{-6}$, while the Euler scheme for simulation uses the finer mesh size $h/4$. In order to reach the stationary distribution, the process is simulated for a burn-in period of $T=5$. We fix the bandwidth at $b=0.5$ and choose $u=(Tbh^2)^{-0.07}$. To estimate the drift, we use the odd design function $f(x)=\int_{0}^{x} \exp(-y^2)dy$ which satisfies (F'). For the jump activity estimator, we use 
\begin{align}
	f(x) &= \begin{cases}
		\exp\left( -\frac{1}{|x|-0.1} \right),& |x|\geq 0.1\\
		0,& |x|<0.1.
	\end{cases} \label{eqn:f-JA}
\end{align}
Figure \ref{fig:MC-alpha} shows the behavior of the estimator $\hat{\alpha}_n(x,\gamma)$ for $\gamma=2$, based on $10^4$ Monte Carlo samples. The dashed lines represent pointwise $25\%$, $50\%$ and $75\%$ quantiles. We find that the asymptotic standard deviation derived in Theorem \ref{thm:alpha-clt} yields confidence sets which match the simulated quantiles reasonably well. However, even for the large sample size of roughly $10^7$ observations, the confidence bands are rather wide. To reduce the standard deviation, we could increase the scaling parameter $u$ of our estimator, at the price of increasing the bias. The results for $u=(Tbh^2)^{-0.08}$ are shown in the right panel of Figure \ref{fig:MC-alpha}, displaying smaller confidence bands but a visible bias. This shows the sensitivity of the estimate with respect to the tuning parameter $u$, as already discussed in section \ref{sec:stable-like}.

\begin{figure}[h]
	\centering
	\includegraphics[width=0.48\textwidth]{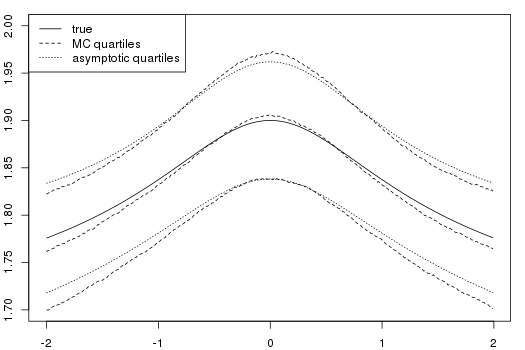}\includegraphics[width=0.48\textwidth]{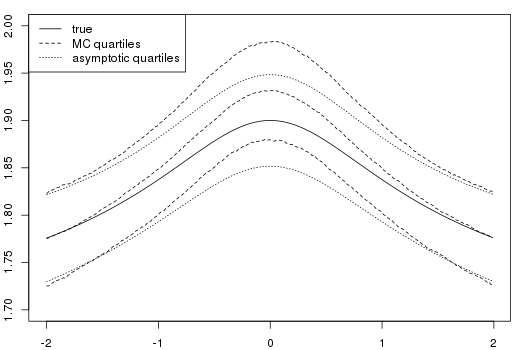}
	\caption{Behavior of $\hat{\alpha}_n(x, 2)$ in comparison with the true value $\alpha(x)$, for the sampling scheme $T=10, h=10^{-6}$. Quartiles are based on $10^4$ Monte Carlo samples. Left: $u=(Tbh^2)^{-0.07}\approx 6.2$, right: $u=(Tbh^2)^{-0.08}\approx 8.0$.}
	\label{fig:MC-alpha}
\end{figure}

A second tuning parameter of our estimator is the factor $\gamma\neq 1$, which affects the asymptotic variance.
The relevant variance factor is given by 
\begin{align*}
	s^2(\gamma,\alpha,f) = \frac{1}{(\log \gamma)^2} \int \frac{(f(z)-\gamma^{-\alpha}f(\gamma z))^2}{|z|^{1+\alpha}}\, dz.
\end{align*} 
In Figure \ref{fig:sgamma}, we depict the value of $s^2(\gamma, \alpha,f)$ as a function of $\alpha$ and $\gamma$, for $f$ as specified in \eqref{eqn:f-JA} below.
For the depicted range of values, increasing $\gamma$ decreases the asymptotic variance. 
To assess the finite sample effect of $\gamma$, we repeat the estimation of $\alpha(x)$ with $\gamma=5$, see Figure \ref{fig:MC-alpha-gammalarge}.
We find that, for the same value $u=(Tbh^2)^{-0.07}$, the estimator incurs a severe bias. 
To reduce the bias, we need to choose a smaller value $u$, which in turn increases the variance of the estimator. 
In the right plot of Figure \ref{fig:MC-alpha-gammalarge}, we show the performance of the estimator for the smaller $u=(Tbh^2)^{-0.05}$, which is very close to the left plot in Figure \ref{fig:MC-alpha}.
This demonstrates that the choice of $\gamma$ should not be based on the value of $s^2(\gamma, \alpha, f)$ alone.

\begin{figure}[tb]
\centering
\includegraphics[width=0.8\textwidth]{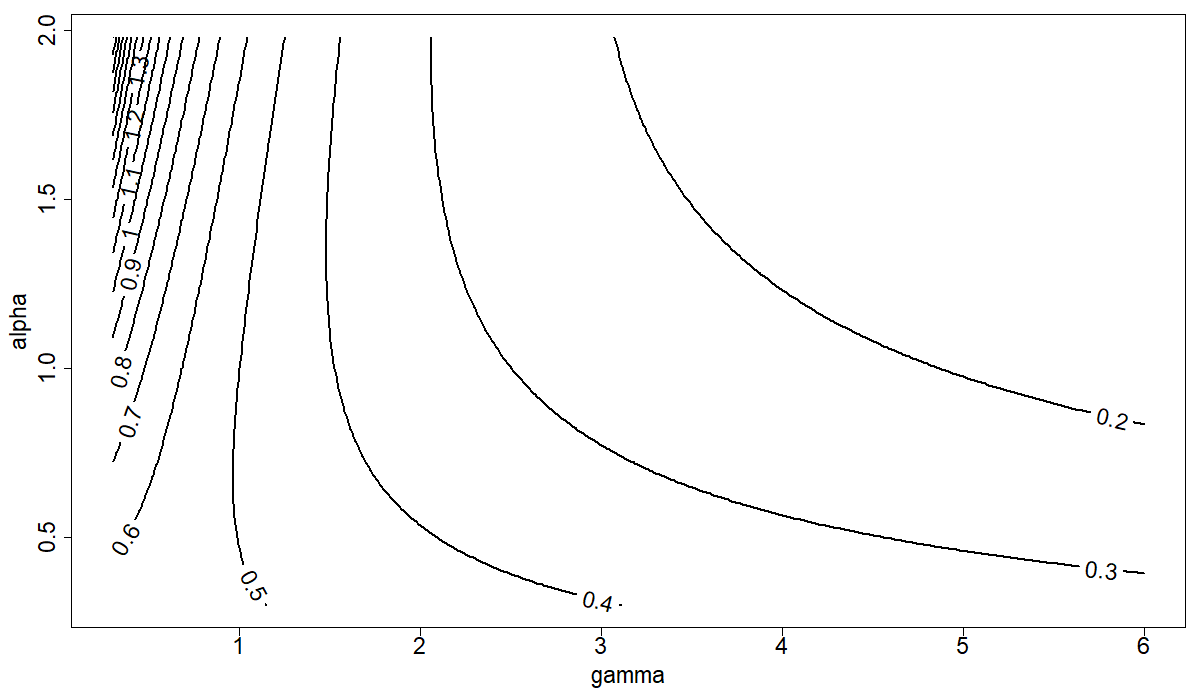}
\caption{Contour plot of the asymptotic variance $s^2(\gamma,\alpha, f)$, as a function of $\gamma$ and $\alpha$, for $f$ as in \eqref{eqn:f-JA}.}
\label{fig:sgamma}
\end{figure}

\begin{figure}[h]
	\centering
	\includegraphics[width=0.48\textwidth]{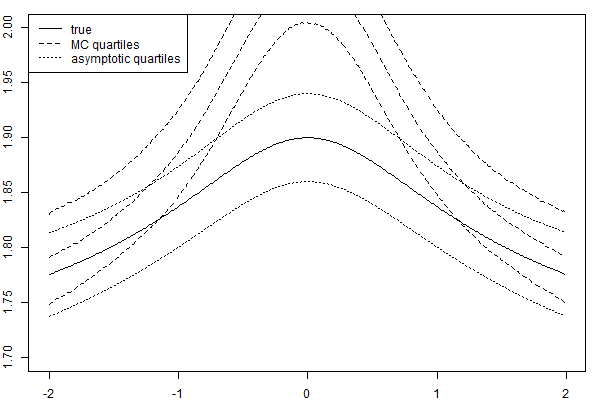}\includegraphics[width=0.48\textwidth]{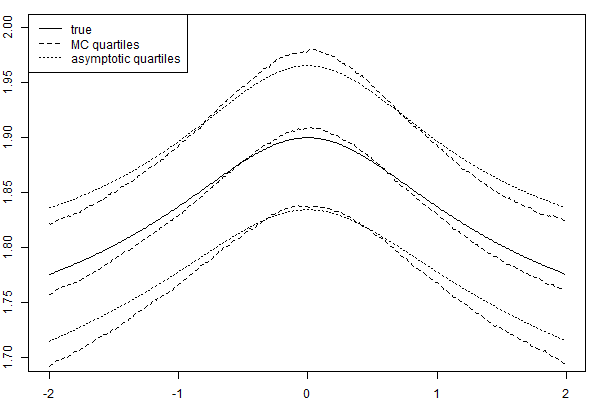}
	\caption{Behavior of $\hat{\alpha}_n(x, 5)$ in comparison with the true value $\alpha(x)$, for the sampling scheme $T=10, h=10^{-6}$. Quartiles are based on $10^4$ Monte Carlo samples. Left: $u=(Tbh^2)^{-0.07}\approx 6.2$, right: $u=(Tbh^2)^{-0.05}\approx 3.7$.}
	\label{fig:MC-alpha-gammalarge}
\end{figure}

In the same sampling scheme, the Monte Carlo results for the drift estimator $\hat{\mu}_n(x)$, i.e.\ with scaling factor $u=1$, are depicted in Figure \ref{fig:MC-mu} (left panel). It is found that the asymptotic distribution approximately matches the empirical performance. However, the estimator has a high variance in this sampling situation. Since the variance of the drift estimator depends on the time horizon $T$, we repeat the simulation for the sampling scheme $T=100, h=10^{-4}$ (Figure \ref{fig:MC-mu}, right panel). Though the total number of observations in this sampling scheme is smaller, the drift estimator is found to be more accurate.  

The variance of the drift estimator can also be decreased by employing the jump-filtered estimator $\hat{\mu}_n^*(x)$. 
In Figure \ref{fig:MC-mu-filtered}, we present the simulation results for $\hat{\mu}_n^*(x)$ with scaling factors $u=10$ and $u=100$. 
Compared to the case $u=1$ in Figure \ref{fig:MC-mu}, as expected, we find that the confidence bands are tighter.
Furthermore, the match with the asymptotic distribution is slightly better.
However, using the larger scaling factor $u=100$ induces a visible bias. 
It would thus be of interest to develop a data-driven procedure for the choice of $u$, which is however out of the scope of this article.

\begin{figure}[h]
	\centering
	\includegraphics[width=0.48\textwidth]{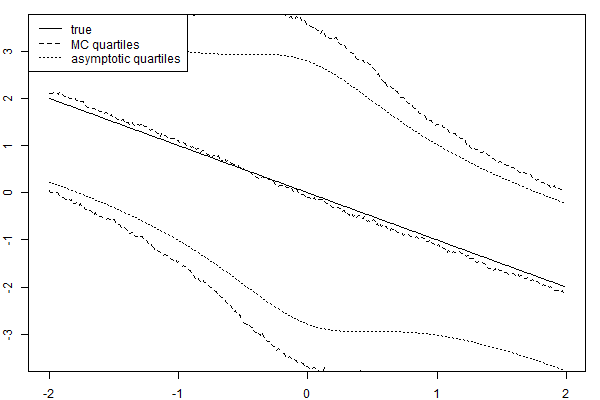}\includegraphics[width=0.48\textwidth]{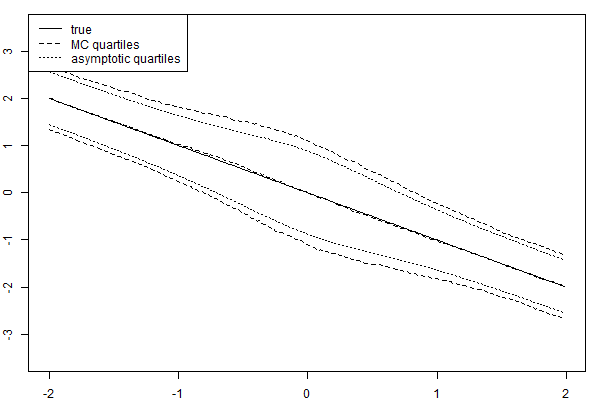}
	\caption{Behavior of $\hat{\mu}_n(x)$ in comparison with the true value $\mu(x)$, for the sampling schemes $T=10, h=10^{-6}$ (left) and $T=100, h=10^{-4}$ (right). Quartiles are based on $10^4$ Monte Carlo samples.}
	\label{fig:MC-mu}
\end{figure}

\begin{figure}[h]
	\centering
	\includegraphics[width=0.48\textwidth]{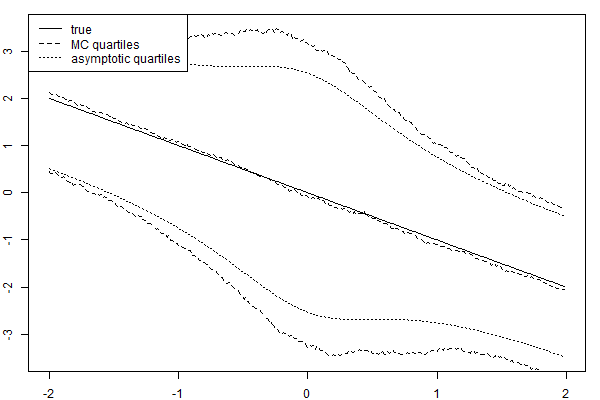}\includegraphics[width=0.48\textwidth]{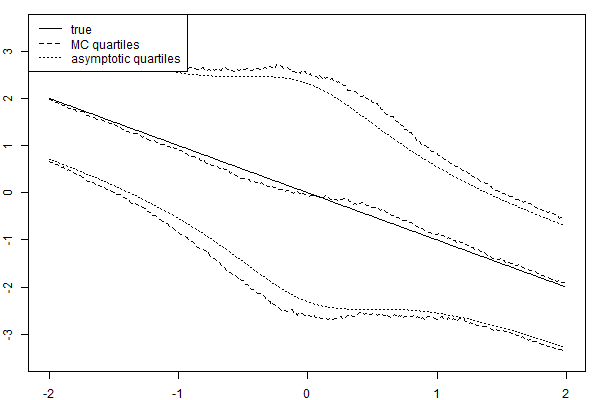}
	\caption{Behavior of the jump-filtered estimator $\hat{\mu}_n^*(x)$ in comparison with the true value $\mu(x)$, for the sampling scheme $T=10, h=10^{-6}$. The scaling factors for the jump filtering are $u=10$ (left) and $u=100$ (right). Quartiles are based on $10^4$ Monte Carlo samples.}
	\label{fig:MC-mu-filtered}
\end{figure}

\section{Proofs}\label{sec:proofs}

\begin{proof}[Proof of Proposition \ref{prop:cond-approx}]
	The bound on the conditional expectation is an immediate consequence of \eqref{eqn:taylor-1}. The conditional variance may be handled by treating $E(f^2(X_{t+h}-X_t)|X_t=x)$ analogously, and noting that \begin{align*}
		E(f(X_{t+h}-X_t)|X_t=x)^2 & \leq (h|\mathcal{A} \tau_xf(x)| + h^2\|\mathcal{A}^2\tau_xf\|_{\infty})^2 \\
		&\leq 2 h^2 |\mathcal{A}^*f(x)| + 2 h^4\|\mathcal{A}^2\tau_xf\|_{\infty}^2.
	\end{align*}
\end{proof}

\subsection{Analytical bounds on the Markov generator}

We will make use of the following properties of the family of norms $\|\cdot\|_{\infty, p}$. The proof of the following Lemma also contains the most technical part of the proof of Theorem \ref{thm:A2full}. Part 3 is not used in the following proofs, but it might still be of interest.

\begin{lemma}\label{lem:normsup}
	Let $f:\R\to\R$.
	\begin{enumerate}
		\item For any $a>0$, let $\kappa_a f(x)=\sup_{y: |y-x|\leq a} |f(y)|$. Then for all $p\geq 0$,
		\begin{align*}
		\|\kappa_af\|_{\infty, p} \leq 2^p (a\vee 1)^p \|f\|_{\infty, p}.
		\end{align*}
		\item For any $p>1$ and $q\in [0,p]$,  \begin{align}
		\label{eqn:intbound}\left\| x\mapsto \int_{-\infty}^{\infty} f(x+z) (|z| \vee 1)^{-q} dz\right\|_{\infty, q} \leq K_{p}  \|f\|_{\infty, p},
		\end{align}
		where $K_p = \frac{2^{p+4} p}{p-1}$.
		\item If $f$ is differentiable, then for any $q>0,\eta>1$ such that $p<q\leq \eta$, \begin{align}
				\label{eqn:intboundderiv}\left\| x\mapsto \int_{-\infty}^{\infty} \left[f(x+z)-f(x)\right] (|z| \vee 1)^{-1-q} dz\right\|_{\infty, p} \leq K^*_{q,p,\eta} \|f'\|_{\infty,\eta},
			\end{align}
			where $K^*_{q,p,\eta}=\left( 1+\frac{1}{q-p} \right) K_{\eta} $
		\item Denote $\tau_rf(x) = f(x-r)$ and $\rho_rf(x)=f(rx)$. Then for $p\geq 0$, we have $\|\tau_r f\|_{\infty, p} \leq 2^p (|r|\vee 1)^p \|f\|_{\infty, p}$ for $r\in\R\setminus\{0\}$, and $\|\rho_r f\|_{\infty, p} \leq (|r|\wedge 1)^{-p}\|f\|_{\infty, p}$.
	\end{enumerate}
\end{lemma}
\begin{proof}[Proof of Lemma \ref{lem:normsup}]
	\underline{Claim 1:}
	Denote $a^*=a\vee 1$. For any $x\in\R$,
	\begin{align*}
	\sup_{y:|y-x|\leq a} |f(y)| \left( |x|\vee 1 \right)^p & \leq \sup_{y:|y-x|\leq a} |f(y)| \left( (|y|+a)\vee 1 \right)^p\\
	&\leq \sup_y |f(y)| \left( (|y|+a^*)\vee 2a^* \right)^p \\
	&\leq \sup_y |f(y)| \left( 2|y|\vee 2a^* \right)^p \\
	&\leq (2a^*)^p \sup_y |f(y)| \left( |y|\vee 1 \right)^p.
	\end{align*}
	In the last step, we used that $a^*\geq 1$.
	
	\underline{Claim 2:}
	For the second claim, as the definition of $\|\cdot \|_{\infty, q}$ is symmetric in $x$, we assume $x>0$. 
	Furthermore, assume without loss of generality $x\geq 2$, since the decay claimed by \eqref{eqn:intbound} is vacuous for small $x$. 
	Indeed, 
	\begin{align*}
		\sup_{x\in[-2,2]} \left|\int f(x+z) (|z|\vee 1)^{-q}\, dz \right| (|x|\vee 1)^q
		&\leq 2^q \|f\|_{\infty, p} \int (|x+z| \vee 1)^{-p}\, dz \\
		&= 2^q \|f\|_{\infty, p} \int (|z| \vee 1)^{-p}\, dz \\
		&= 2^{q+1} \|f\|_{\infty, p} \frac{p}{p-1} \leq K_p \|f\|_{\infty, p}.
	\end{align*}
	
	For $x\geq 2$, we decompose the integral as \begin{align}
		&\quad \left| \int f(x+z) (|z|\vee 1)^{-q} dz \right|\nonumber\\
		&\leq \int_{-1}^1 |f(x+z)| (|z|\vee 1)^{-q} dz \nonumber\\
		&\qquad  + \int_{1}^\infty \ldots dz  + \int_{-x+1}^{-1} \ldots dz + \int_{-x-1}^{-x+1}\ldots dz +\int_{-\infty}^{-x-1}\ldots dz \nonumber\\
		&=I_1 + I_2 + I_3+I_4+I_5 \label{eqn:normsup-decomp}
	\end{align}
	The first term may be bounded as $I_1 \leq \sup_{y:|x-y|\leq 1} |f(y)| \leq 2^q \|f\|_{\infty, q} (|x|\vee 1)^{-q}$, by virtue of Claim 1.
	For the second term may be bounded as 
	\begin{align*}
		I_2 
		& \leq \int_1^\infty |z|^{-q} |x+z|^{-p} \|f\|_{\infty, p}\, dz\\
		&\leq \|f\|_{\infty, p}\int_1^\infty |z|^{-q} |x+z|^{-p+q} (|x|\vee 1)^{-q}\, dz\\ 
		&\leq \|f\|_{\infty, p}(|x|\vee 1)^{-q} \int_1^\infty |z|^{-p}\, dz \\
		&= \|f\|_{\infty, p}\frac{(|x|\vee 1)^{-q}}{p-1},
	\end{align*}
	since $|x+z|\geq |z|$ in this situation.
	Moreover, since $x\geq 2$,
	\begin{align*}
		I_4
		&\leq \|f\|_{\infty, 0}\,\int_{-x-1}^{-x+1} |z|^{-q}\, dz\\
		&\leq \|f\|_{\infty, 0}\,\int_{-x-1}^{-x+1} (|x|/2)^{-q}\, dz\\
		&\leq \|f\|_{\infty, 0}\,2^{q+1} (|x|\vee 1)^{-q}.
	\end{align*}
	By substitution, and since $p>1$, the integral $I_5$ may be bounded as 
	\begin{align*}
		I_5
		&\leq \int_{-\infty}^{-x-1} |z|^{-q} |x+z|^{-p}\|f\|_{\infty, p}\, dz \\
		&=\|f\|_{\infty, p}|x|^{1-p-q} \int_{-\infty}^{-1-\frac{1}{x}} |z|^{-q} |z+1|^{-p} dz \\ 
		& \leq \|f\|_{\infty, p}|x|^{1-p-q}\int_{-\infty}^{-\frac{1}{x}} |z|^{-p} dz 
		 = \|f\|_{\infty, p}\frac{|x|^{-q}}{p-1} .
	\end{align*}
	Finally, since $q<p$,
	\begin{align*}
		I_3
		&\leq \int_{-x+1}^{-1} |z|^{-q} |x+z|^{-p} \|f\|_{\infty, p}\, dz\\
		&=\|f\|_{\infty, p}|x|^{1-p-q}\int_{-1+\frac{1}{x}}^{-\frac{1}{x}} |z|^{-q} \left|z+1\right|^{-p} dz \\
		&\leq \|f\|_{\infty, p}|x|^{1-p-q} \left(2^q\int_{-1+\frac{1}{x}}^{-\frac{1}{2}} |z+1|^{-p} dz + 2^p\int_{-\frac{1}{2}}^{-\frac{1}{x}} |z|^{-q} dz\right) \\
		&\leq \|f\|_{\infty, p}|x|^{1-p-q} 2^p \left(\int_{\frac{1}{x}}^{\infty} |z|^{-p}\, dz + \int_{-\infty}^{-\frac{1}{x}} |z|^{-p}\, dz\right)\\
		&= \|f\|_{\infty, p}\frac{2^{p+1}}{p-1} |x|^{-q}.
	\end{align*}
	Plugging these estimates into \eqref{eqn:normsup-decomp}, we find that \begin{align*}
		&\quad \left|\int f(x+z) (|z|\vee 1)^{-q} dz\right| \\
		&\leq 2^{q+1} \|f\|_{\infty, q} (|x|\vee 1)^{-q} + \|f\|_{\infty, p} \frac{(|x|\vee 1)^{-q}}{p-1}+ \|f\|_{\infty, 0} 2^{q+1} (|x|\vee 1)^{-q} \\
		&\qquad  + \|f\|_{\infty, p} |x|^{-q} \frac{2^{p+1}}{p-1} + \|f\|_{\infty, p} \frac{|x|^{-q}}{p-1}\\
		&\leq  \|f\|_{\infty, p} (|x|\vee 1)^{-q} \frac{2^{p+4}}{p-1}.
	\end{align*}
	Note that we may replace $|x|$ by $(|x|\vee 1)$ since we assumed without loss of generality that $x>2$.
	This completes the proof of the second statement.
	
	\underline{Claim 3:}
	The third inequality can be handled by Fubini's theorem via\begin{align*}
		\left|\int_0^\infty \left[f(x+z)-f(x)\right] \left( |z|\vee 1 \right)^{-1-q}\, dz \right| 
		&\leq \int_0^\infty \int_0^z |f'(x+r)| \left( |z|\vee 1 \right)^{-1-q} dr\, dz \\
		&=\int_0^\infty \int_r^\infty |f'(x+r)| \left( |z|\vee 1 \right)^{-1-q}  dz\, dr \\
		&\leq \int_0^\infty |f'(x+r)| \left( |r| \vee 1 \right)^{-p}\int_r^\infty \left( |z|\vee 1 \right)^{-1-q+p} dz\, dr\\
		&\leq \left( 1+\frac{1}{q-p} \right) \int_0^\infty |f'(x+r)| \left( |r|\vee 1 \right)^{-p}\,dr ,
	\end{align*}
	since $q>p$ in this situation.
	By symmetry in $z$, this bound may be extended to 
	\begin{align*}
		\left|\int_{-\infty}^\infty \left[f(x+z)-f(x)\right] \left( |z|\vee 1 \right)^{-1-q}\, dz \right| 
		& \leq  \left( 1+\frac{1}{q-p} \right) \int_{-\infty}^\infty |f'(x+r)| \left( |r|\vee 1 \right)^{-p}\,dr.
	\end{align*} 
	Together with \eqref{eqn:intbound}, we conclude that \eqref{eqn:intboundderiv} holds.

	\underline{Claim 4:}
	The fourth claim of the Lemma can be checked as \begin{align*}
		\sup_x |f(x-r)| \left( |x|\vee 1 \right)^p &= \sup_x |f(x)| \left( |x+r|\vee 1 \right)^p \\
		&\leq \sup_x |f(x)| \left( (|x|\vee 1) + (|r|\vee 1) \right)^p\\
		&\leq (|r|\vee 1)^p \sup_x |f(x)|  \left( (|x|\vee 1) + 1 \right)^p \\
		&\leq  2^p(|r|\vee 1)^p \sup_x |f(x)|  \left( |x|\vee 1 \right)^p.
	\end{align*}
	Furthermore, $\sup_x |f(rx)| \left( |x| \vee 1 \right)^p \leq \sup_x |f(rx)| \left( \frac{|rx| \vee 1}{|r|\wedge 1} \right)^p=(|r|\wedge 1)^{-p} \|f\|_{\infty, p}$.
\end{proof}

To simplify notation, we write the generator of $X_t$ as 
\begin{align*}
	\mathcal{A} f(x) = \mu(x) f'(x) + \frac{\sigma^2(x)}{2} f''(x) + \mathcal{J}f(x),
\end{align*}
where $\mathcal{J}$ is generator corresponding to the jump part, i.e.\ \begin{align*}
	\mathcal{J}f (x) 
	&= \int_\R f(x+c(x,z)) - f(x) - c(x,z) f'(x) \mathds{1}_{|c(x,z)|\leq 1} \nu(dz) \\
	&= \int_\R \left[f(x+z) - f(x) - z f'(x) \mathds{1}_{|z|\leq 1}\right] \rho(x,z)\, dz.
\end{align*}

In order to derive boundedness properties of $\mathcal{A}f$, We consider the jump operator $\mathcal{J}$ first. The principal technical result which enables us to derive bounds on $\|\mathcal{A} f\|_{\infty,0}$ is the following lemma. Although we will only need derivatives up to the second order for our statistical purposes, we formulate the result slightly more generally. This allows us to derive the stronger Theorem \ref{thm:Ak} by imposing stricter assumptions. To this end, we introduce the following stricter versions of assumption (J2), for any integer $m\geq 2$. 

\begin{itemize}
	\item[\textbf{(J2-m)}] The jump density $\rho(x,z)$ satisfies $|\frac{d^k}{dx^k}\rho(x,z)| \leq C_\rho \left( |z|^{-1-\alpha} \vee |z|^{-1-\tau} \right)$ for some $1<\tau<\alpha<2$ and $k=0,1,\ldots, m$.
\end{itemize}

In particular, (J2-2) is equivalent to (J2).

\begin{lemma}\label{lem:C2domain}
	Under assumptions (J1) and (J2-m), for any $q<\tau$, there exists a constant $K$ such that for any function $f$, and any $k\in\N, k\leq m$,
	\begin{align}
		\label{eqn:J0bound}\left\|\mathcal{J} f \right\|_{\infty, 0} 	&\leq  K \left(\|f''\|_{\infty, 0}+  \|f'\|_{\infty, 0}  \right) \\
		\label{eqn:J1bound}\left\|\mathcal{J} f \right\|_{\infty, q} 	&\leq  K \left(\|f''\|_{\infty, q}+  \|f'\|_{\infty, \tau}  \right),\\
		\label{eqn:Jkbound} \left\| (\mathcal{J} f)^{(k)} \right\|_{\infty, 0} &\leq K \sum_{j=1}^{k+2} \|f^{(j)}\|_{\infty, 0},\\
		\label{eqn:J2bound} \left\| (\mathcal{J} f)^{(k)} \right\|_{\infty, q} &\leq K (\|f^{(k+2)}\|_{\infty, q} +  \sum_{j=1}^{k+1} \|f^{(j)}\|_{\infty, \tau}).
	\end{align}
	If either \eqref{eqn:Jkbound} or \eqref{eqn:J2bound} is finite, then $x\mapsto (\mathcal{J}f)^{(k)}(x)$ is continuous and hence $\mathcal{J}f \in \mathcal{C}^k$.
	The constant $K$ depends on the values of $\alpha, C_\rho,\tau,$ and $q$. 
\end{lemma}
\begin{proof}[Proof of Lemma \ref{lem:C2domain}]
	We may assume that all occurring norms $\|f^{(j)}\|_{\infty, q}$ are finite, since otherwise there is nothing to show.
	Using a Taylor expansion of $f$ and the assumptions imposed upon $\rho$, we decompose the integral as
	\begin{align}
		&\quad \nonumber\left|\mathcal{J} f(x) \right|	\\
		\nonumber&\leq  \int_{[-1,1]^c} \left| f(x+z)-f(x) \right|\rho(x,z) dz + \int_{-1}^{1} |z|^2 \rho(x,z)dz \sup_{|y-x|\leq 1} |f''(y)| \\
		\label{eqn:J1decomp}& \leq C_\rho\int |f(x+z)-f(x)| \left( |z| \vee 1 \right)^{-1-\tau} dz + C_\rho \int_{-1}^1 |z|^{1-\alpha}dz\sup_{|y-x|\leq 1} |f''(y)|  \\
		\nonumber&\leq C_\rho K^*_{\tau,q,\tau} \|f'\|_{\infty, \tau} \left(|x|\vee 1 \right)^{-q} + 2^qC_\rho \frac{2}{2-\alpha} \|f''\|_{\infty, q} \left(|x|\vee 1 \right)^{-q}.
	\end{align}
	The latter inequality holds by virtue of Lemma \ref{lem:normsup}(3) since $\tau>1$. 
	This establishes \eqref{eqn:J1bound}. 
	For $q=0$, we may also obtain the bound 
	\begin{align*}
		\int_{[-1,1]^c} |f(x+z)-f(x)| \rho(x,z) dz &\leq 2 C_\rho\|f'\|_{\infty, 0} \int_{1}^\infty |z|^{-\tau}dz < \infty,
	\end{align*}
	proving \eqref{eqn:J0bound}.
	
	The first derivative can be written as
	\begin{align}
		\begin{split}
		\frac{d}{dx} \mathcal{J}f(x) &= \int \left[f'(x+z) - f'(x) - z f''(x) \mathds{1}_{|z|\leq 1} \right] \rho(x,z)dz\\
		&\qquad + \int \left[f(x+z) - f(x) - z f'(x) \mathds{1}_{|z|\leq 1} \right] \frac{d}{dx} \rho(x,z) dz
		\label{eqn:jumpderiv-1}\end{split} \\
		&= I_1(x) + I_2(x).\nonumber
	\end{align}
	Since (J2-m) requires the same bound on $|\rho(x,z)|$ and $|\frac{d}{dx} \rho(x,z)|$, we may apply the same reasoning as in the derivation of \eqref{eqn:J1bound} to bound the integrands of $I_1$ and $I_2$, i.e.\ \begin{align*}
		I_1(x) 
		&\leq K (|x|\vee 1)^{-q}( \|f''\|_{\infty, \tau} + \|f'''\|_{\infty, q}),\\
		I_2(x)
		&\leq K (|x|\vee 1)^{-q}( \|f'\|_{\infty, \tau} + \|f''\|_{\infty, q})
	\end{align*} 
	To justify the exchange of integration and differentiation in \eqref{eqn:jumpderiv-1}, an integrable majorant can be obtained by using a rougher bound of the integrand, such as \begin{align*}
		|I_1(x)| & \leq C_\rho\int_{[-1,1]^c} |f'(x+z)-f'(x)| |z|^{-1-\tau} + C_\rho \sup_x |f'''(x)|\int_{-1}^{1} |z|^{1-\alpha} dz   \\ 
		& \leq C_\rho \int \|f''\|_{\infty, 0} |z| \left( |z| \vee 1 \right)^{-1-\tau}dz + C_\rho  \|f'''\|_{\infty, 0} \int_{-1}^1 |z|^{1-\alpha} dz < \infty.
	\end{align*}
	The integrand of $I_2(x)$ can be bounded analogously.
	
	In the same spirit, we can find an integrable majorant for the higher order derivatives, which are given by
	\begin{align}
			\label{eqn:int-k-deriv}\left|\frac{d^k}{dx^k} \mathcal{J}f(x)\right| &= \left|\sum_{j=0}^k \binom{k}{j} \int \left[ f^{(j)}(x+z) - f^{(j)}(x) - zf^{(j+1)}(x) \mathds{1}_{|z|\leq 1} \right] \frac{d^{k-j}}{dx^{k-j}}\rho(x,z)dz \right| \\
			\nonumber&\leq (|x|\vee 1)^{-q} \sum_{j=0}^k \binom{k}{j} C_\rho \left[ K^*_{\tau,q,\eta} \|f^{(j+1)}\|_{\infty, \eta} + 2^q \|f^{(j+2)}\|_{\infty, q} \right].
	\end{align}
	Since $\|f\|_{\infty, q} \leq \|f\|_{\infty, \eta}$ for any function $f$ as $q\leq \eta$, we conclude that \eqref{eqn:J2bound} holds. Furthermore, since the integrand in \eqref{eqn:int-k-deriv} is continuous in $x$ and has a majorant, the integral is also continuous in $x$. 
	
	By applying the same technique as in \eqref{eqn:J0bound} to the summands in \eqref{eqn:int-k-deriv}, we also obtain the bound \eqref{eqn:Jkbound}.
\end{proof}
	
The full generator $\mathcal{A}f = \mu f' + \frac{\sigma^2}{2} f'' + \mathcal{J}f$ can be bounded similarly by making use of Lemma \ref{lem:C2domain}. 
The formulation of the following corollary yields bounds on higher order derivatives if assumptions stronger than (D) and (V) are satisfied. 
Note that for $m=2$, the additional condition is redundant.

\begin{cor}\label{cor:full-gen}
Let assumptions (J1), (J2-m), (D) and (V) hold. Moreover, suppose that $\mu,\sigma^2\in\mathcal{C}^m$ such that $\|(\sigma^2)^{(k)}\|_{\infty, -p_V} \leq C_\sigma, \|\mu^{(k)}\|_{\infty, -p_D} \leq C_\mu$ for all $k=1,\ldots, m$. Then for any $0\leq q<\tau$, there exists a constant $K$ such that for any function $f:\R\to\R$ which is $k+2$ times differentiable,
\begin{align}
	\label{eqn:fullgen-infty-0}\|\mathcal{A}f\|_{\infty, 0} & \leq K \left( \|f'\|_{\infty,1} + \|f''\|_{\infty, 0} \right)\\
	\label{eqn:fullgen-deriv}\|(\mathcal{A}f)^{(k)}\|_{\infty, q} & \leq K \left( \|f^{(k+2)}\|_{\infty, q} + \sum_{j=1}^{k+1} \|f^{(j)}\|_{\infty, \chi(q)} \right),
\end{align}
for $\chi(q)=q+\max(\tau, p_D, p_V)$.
The constant $K$ depends on the values of $\alpha, C_\rho,C_\mu,C_\sigma,\tau$ and $q$.
\end{cor}
\begin{proof}[Proof of Corollary \ref{cor:full-gen}]
Inequality \eqref{eqn:fullgen-infty-0} is a consequence of Lemma \ref{lem:C2domain} and the duality \eqref{eqn:norm-duality}, since
\begin{align*}
	\|\mathcal{A}f\|_{\infty, 0} & \leq \|\mu f'\|_{\infty,0} + \frac{1}{2} \|\sigma^2 f''\|_{\infty,0} + \|\mathcal{J}f\|_{\infty, 0} \\ 
	&\leq C_\mu \|f'\|_{\infty, 1} + C_\sigma \|f''\|_{\infty, 0} + K(\|f''\|_{\infty, 0} + \|f'\|_{\infty,0}).
\end{align*}
To treat the derivatives, we write
\begin{align*}
	\frac{d^k}{dx^k} \mathcal{A}f(x) &= \frac{d^k}{dx^k} \left[ \mu(x)f'(x) + \frac{\sigma^2(x)}{2}f''(x) + \mathcal{J}f(x) \right] \\
	&= (\mathcal{J}f)^{(k)}(x) + \mu(x)f^{(k+1)}(x) + \frac{\sigma^2(x)}{2} f^{(k+2)}(x) \\
	&\qquad +\sum_{j=1}^k \binom{k}{j} \mu^{(j)}(x) f^{(k-j+1)}(x) + \frac{1}{2}\sum_{j=1}^k \binom{k}{j} (\sigma^2)^{(j)}(x) f^{(k-j+2)}(x).
\end{align*}
By bounding each term individually via the duality \eqref{eqn:norm-duality}, and applying Lemma \ref{lem:C2domain} with $\eta=q+\tau$, this yields \begin{align*}
	\|(\mathcal{A}f)^{(k)}\|_{\infty, q} & \leq K \left( \|f^{(k+2)}\|_{\infty, q} + \sum_{j=1}^{k+1} \|f^{(j)}\|_{\infty,q+\tau} + \|f^{(k+1)}\|_{\infty, q+1}  \right) \\
	&\qquad + K \left( \sum_{j=1}^k \|f^{(k-j+1)}\|_{\infty, q+p_D} + \sum_{j=1}^k \|f^{(k-j+2)}\|_{\infty, q+p_V}\right)\\
	&\leq K \left( \|f^{(k+2)}\|_{\infty, q} + \sum_{j=1}^{k+1} \|f^{(j)}\|_{\infty, \chi(q)} \right),
\end{align*}
where we let $K$ vary from line to line.
\end{proof}

By applying \eqref{eqn:fullgen-infty-0} to $\mathcal{A}f$, we are now able to complete the proof of Theorem \ref{thm:A2full}.

\begin{proof}[Proof of Theorem \ref{thm:A2full}]
	Let $K$ denote a generic constant, not depending on $f$, which may change from line to line.
	By applying \eqref{eqn:fullgen-infty-0} and \eqref{eqn:fullgen-deriv}, we obtain
	\begin{align*}
		\|\mathcal{A}^2 f\|_{\infty, 0} & \leq K \left( \|(\mathcal{A}f)'\|_{\infty, 1} + \|(\mathcal{A}f)''\|_{\infty, 0} \right) \\
		&\leq K \left( \|f'''\|_{\infty,1} + \|f''\|_{\infty,\chi(1)} + \|f'\|_{\infty,\chi(1)}  \right) \\
		&\qquad + K \left( \|f''''\|_{\infty,0} + \|f'''\|_{\infty,\chi(0)}+ \|f''\|_{\infty,\chi(0)} + \|f'\|_{\infty,\chi(0)}  \right) \\
		&\leq K \left( \|f''''\|_{\infty, 0} + \|f'''\|_{\infty,\chi(0)} + \|f''\|_{\infty,\chi(1)} +\|f'\|_{\infty,\chi(1)} \right)
	\end{align*}
\end{proof}

A very similar reasoning establishes Theorem \ref{thm:Ak}.

\begin{proof}[Proof of Theorem \ref{thm:Ak}]
Let $K$ denote a generic constant, not depending on $x$ or $f$, which may change from line to line. 
The stronger assumptions of Theorem \ref{thm:Ak} imply that condition (J2-2m) holds. 
By Lemma \ref{lem:C2domain}, for any $r\leq 2m$, \begin{align*}
	|(\mathcal{A} f)^{(r)}(x)| &\leq |(\mathcal{J}f)^{(r)}(x)| + \sum_{j=0}^r \binom{r}{j}\left|\mu^{(j)}(x) f^{(r-j+1)}(x)\right|+ \frac{1}{2} \sum_{j=0}^r \binom{r}{j}\left|(\sigma^2)^{(j)}(x) f^{(r-j+2)}(x)\right| \\
	&\leq \|(\mathcal{J}f)^{(r)}\|_{\infty, 0} + 2^r(C_\mu + C_\sigma) \sum_{j=1}^{r+2} \|f^{(j)}\|_{\infty, 0} \\
	&\leq K \sum_{j=1}^{r+2} \|f^{(j)}\|_{\infty, 0}.
\end{align*}
Choosing $r=0$, this proves the theorem for $m=0$. Now suppose that the claim of the theorem holds for a fixed value $m$, and that conditions of the theorem hold for all $k\leq 2(m+1)$. Then, \begin{align*}
	\|\mathcal{A}^{m+1} f\|_{\infty, 0} = \|\mathcal{A}^{m} (\mathcal{A} f)\|_{\infty,0} &\leq K \sum_{j=1}^{2m} \| (\mathcal{A} f)^{(j)} \|_{\infty, 0} \\
	&\leq K \sum_{j=1}^{2m} \sum_{l=1}^{j+2} \|f^{(l)}\|_{\infty, 0} \qquad \leq K \sum_{j=1}^{2(m+1)} \|f^{(j)}\|_{\infty, 0}.
\end{align*}
By induction, this completes the proof.
\end{proof}

The following simple extension proves useful to control the bias in kernel estimation.

\begin{lemma}\label{lem:genstar-smooth}
Under assumptions (J1), (J2), there exists a constant $K_\rho>0$ such that for all $f\in \mathcal{C}^2$ and $k=0,1,2$, \begin{align*}
	\left|\frac{d^k}{dx^k}\mathcal{J}^*f(x) \right| \leq K_\rho \left( \|f\|_{\infty,0} + \|f''\|_{\infty,0} \right).
\end{align*}
\end{lemma}
\begin{proof}[Proof of Lemma \ref{lem:genstar-smooth}]
We have
\begin{align*}
	\left|\frac{d^k}{dx^k}\mathcal{J}^*f(x)\right|&= \left| \int \left[f(z)-f(0)-zf'(0)\mathds{1}_{|z|\leq 1}\right] \frac{d^k}{dx^k} \rho(x,z)dz\right| \\
	&\leq C_\rho \int |f(z)-f(0)-zf'(0)\mathds{1}_{|z|\leq 1}| \left( |z|^{-1-\alpha}\vee |z|^{-1-\tau} \right)dz \\
	&\leq K_\rho \left(\|f''\|_{\infty,0} + \|f\|_{\infty,0} \right)
\end{align*}
for a constant $K_\rho$ depending on $\rho$. This upper bound also justifies the exchange of differentiation and integration.
\end{proof}

\subsection{Stable-like processes}

\begin{lemma}\label{lem:SL-T-J}
(J1), (SL1) and (SL2) imply (J2).
\end{lemma}
\begin{proof}[Proof of Lemma \ref{lem:SL-T-J}]
	Assume (J1), (SL1) and (SL2). The condition (J2) for $|z|> 1$ clearly holds. For $|z|\leq 1$, we also have $\rho(x,z) \leq  \|r\|_\infty(1 + C_g) |z|^{-1-\alpha}$ for $\overline{\alpha}=\sup_x  \alpha(x)$. Furthermore, for $|z|\leq 1$, \begin{align*}
		\left|\frac{d}{dx} \rho(x,z)\right| &\leq \frac{\left|\frac{d}{dx}(r(x) + r(x)g(x,z))\right|}{|z|^{1+\alpha(x)}} + \frac{|r(x)+r(x)g(x,z)|}{|z|^{1+\alpha(x)}} |\ln|z|| \alpha'(x)\\
		&\leq C \frac{1}{|z|^{-1-\overline{\alpha}}} \ln|z|,
	\end{align*}
	for some finite constant $C$ depending on $C_g, \|r'\|_\infty, \|r\|_\infty, \|\alpha'\|_\infty$. Since $\overline{\alpha}<2$ by (SL2), $|z|^{-1-\overline{\alpha}} |\ln|z|| \leq C |z|^{-1-\gamma}$ for $\gamma=\frac{2+\overline{\alpha}}{2}<2$ and some finite factor $C$, for all $|z|\leq 1$. Analogously, we obtain a bound on $\frac{d^2}{dx^k} \rho(x,z)$, such that (J2) holds. 
\end{proof}

\begin{proof}[Proof of Lemma \ref{lem:locstable-generator}]
Without loss of generality, assume $\delta(x)<\alpha(x)$. Then
\begin{align}
	\nonumber &\quad\left|\mathcal{J}^*f_u(x) -  u^{\alpha(x)} r(x) f^{[\alpha(x)]}(0)\right|\\
	\nonumber &=\left|\mathcal{J}^*f_u(x) -   r(x) f_u^{[\alpha(x)]}(0)\right|\\
	\begin{split} &=\Big|\int_{-1}^1 \left[ f(uz)-f(0)-f'(0)uz\mathds{1}_{|uz|\leq 1} \right]\frac{r(x)g(x,z)}{|z|^{1+\alpha(x)}}dz \\
	&\qquad+ \int_{[-1,1]^c} \left[ f(uz)-f(0) \right] \left[\rho(x,z) - \frac{r(x)}{|z|^{1+\alpha(x)}} \right]dz \Big|. \end{split}\label{eqn:1}
\end{align}
	Note that we changed the threshold of the indicator function, which is possible because $g$ is assumed to be symmetric. 
	Then
\begin{align*}
	(\ref*{eqn:1})
	&\leq  r(x) C_g \int_{-1}^{1}  \frac{\left|f(uz)-f(0) - uz f'(0)\mathds{1}_{u|z|\leq 1} \right|}{|z|^{1+\alpha(x)-\delta(x)}}dz \\
	&\qquad + C_\rho \int_{[-1,1]^c}  \frac{\left|f(uz)-f(0)\right|}{|z|^{1+\tau}}dz + r(x) \int_{[-1,1]^c}  \frac{\left|f(uz)-f(0)\right|}{|z|^{1+\alpha(x)}}dz\\
	&\leq C_g r(x) u^{\alpha(x)-\delta(x)} \int_{-u}^{u} \frac{\left| f(z)-f(0) -z f'(0) \mathds{1}_{|z|\leq 1} \right|}{|z|^{1+\alpha(x)-\delta(x)}}dz + 4 \|f\|_{\infty,0} \left( \frac{C_\rho}{\tau} + \frac{r(x)}{\alpha(x)} \right) \\
	&\leq C_g r(x)u^{\alpha(x)-\delta(x)} \left(\|f''\|_{\infty, 0} \int_{-1}^{1} |z|^{1+\delta(x) -\alpha(x)}dz + 2\|f\|_{\infty, 0}\int_1^{u}|z|^{-1+\delta(x)-\alpha(x)}\,dz \right) \\
	&\qquad  + 4 \|f\|_{\infty,0} \left( \frac{C_\rho}{\tau} + \frac{r(x)}{\alpha(x)} \right) \\
	&\leq C_g r(x)u^{\alpha(x)-\delta(x)}\left( \frac{2\|f''\|_{\infty, 0}}{2-\alpha(x)} + \frac{2\|f\|_{\infty,0}}{\alpha(x)-\delta(x)} \right) +  4 \|f\|_{\infty,0} \left( \frac{C_\rho}{\tau} + \frac{r(x)}{\alpha(x)-\delta(x)} \right).
\end{align*}
Since the terms $\tau,2-\alpha(x),\alpha(x)-\delta(x)<2$ are all bounded, and $\tau>1$, we can bound the latter expression by \begin{align*}
	\tilde{C}_\rho (\|f''\|_\infty + \|f\|_\infty) \frac{r(x)+1}{(2-\alpha(x)) (\alpha(x)-\delta(x))} u^{\alpha(x)-\delta(x)},
\end{align*}
as claimed in the Lemma.
\end{proof}

\begin{proof}[Proof of Theorem \ref{thm:cond-approx-SL}]
	Note that $\mathcal{A}^*f_u(x) = \mathcal{J}^*f_u(x)$ by our choice of $f$ satisfying (F). 
	Then Proposition \ref{prop:cond-approx} and Lemma \ref{lem:locstable-generator} yield
	\begin{align*}
		&\quad \left| E(f_u(X_{t+h}-X_t))|X_t=x) - hu^{\alpha(x)}r(x) f^{[\alpha(x)]}(0)  \right| \\
		&\leq  \left| E(f_u(X_{t+h}-X_t))|X_t=x) - h\mathcal{A}^*f_u(x)  \right| \,+\,  \left| h\mathcal{A}^*f_u(x) - hu^{\alpha(x)}r(x) f^{[\alpha(x)]}(0)  \right| \\
		&\leq h^2\|\mathcal{A}^2\tau_xf_u\|_{\infty} + K_x \tilde{C}_\rho u^{\alpha(x)-\delta(x)}(\|f''\|_{\infty} + \|f\|_{\infty}).
	\end{align*}
	Furthermore, 
	\begin{align*}
		&\quad \left| \Var(f_u(X_{t+h}-X_t)|X_t=x) - hu^{\alpha(x)}r(x) (f^2)^{[\alpha(x)]}(0) \right| \\
		&\leq h^2 \|\mathcal{A}^2\tau_xf_u^2\|_{\infty} + 2h^2 |\mathcal{A}^*f_u(x)|^2 + 2h^4 \|\mathcal{A}^2\tau_xf_u\|_{\infty}\\
		&\qquad+ K_x \tilde{C}_\rho u^{\alpha(x)-\delta(x)}(\|(f^2)''\|_{\infty} + \|f^2\|_{\infty})
	\end{align*}
	It thus suffices to bound the terms $\|\mathcal{A}^2\tau_x f_u\|_\infty, \|\mathcal{A}^2\tau_x f_u^2\|_\infty$ and $|\mathcal{A}^*f_u(x)| = |\mathcal{J}^*f_u(x)|$.
	
	By Theorem \ref{thm:A2full}, we have for $q=(\tau\vee p_V \vee p_D)+1$, and $u\geq 1$, \begin{align}
		\nonumber\|\mathcal{A}^2\tau_x f_u\|_\infty 	&\leq K \left( \|\tau_xf_u'\|_{\infty,q} + \|\tau_xf_u''\|_{\infty, q}+ \|\tau_xf_u'''\|_{\infty,q-1} + \|\tau_xf_u''''\|_{\infty, 0}\right) \\
		\label{eqn:cond-approx-A2}& \leq u^4 K \left( \|f'\|_{\infty,q} + \|f''\|_{\infty, q}+ \|f'''\|_{\infty,q-1} + \|f''''\|_{\infty, 0}\right) (|x|\vee 1)^{q} 2^{q},
	\end{align}
	where we applied Lemma \ref{lem:normsup} in the last step. 
	The term $u^4$ occurs since $\|f_u''''\|_{\infty, 0} = u^4\|f\|_{\infty, 0}$. 
	The same bound holds for $\|\mathcal{A}^2\tau_x f_u^2\|_\infty$. Moreover, \begin{align}
		\label{eqn:cond-approx-J}|\mathcal{J}^*f_u(x)| 
		&\leq u^{\alpha(x)}r(x) f^{[\alpha(x)]}(0) + |\mathcal{J}^*f_u(x) - u^{\alpha(x)}r(x) f^{[\alpha(x)]}(0)|\\
		\nonumber &\leq u^{\alpha(x)}r(x)f^{[\alpha(x)]}(0) + K_x \tilde{C}_\rho(\|f''\|_\infty + \|f\|_\infty) u^{\alpha(x)-\delta(x)}\\
		\nonumber&\leq \tilde{K}_x\tilde{K}_f u^{\alpha(x)}.
	\end{align}
	To see that the bound of \eqref{eqn:cond-approx-J} factorizes as $\tilde{K}_x \tilde{K}_f$, we recall that $|f^{[\alpha(x)]}(0)|\leq 2\|f\|_\infty/(2-\alpha(x)) + \|f''\|_\infty /\alpha(x)$, as stated prior to Lemma \ref{lem:locstable-generator}. The constant $K_x$ appearing in \eqref{eqn:cond-approx-J}, which origins from Lemma \ref{lem:locstable-generator} can be simplified, since under (SL2), $\alpha$ is bounded away from two and $r(x)$ is upper bounded. Therefore, $\tilde{K}_x$ above is of the form $\tilde{K}_x \propto \frac{1}{\alpha(x)(\alpha(x)-\delta(x))}$.
	
	Due to \eqref{eqn:cond-approx-A2} and \eqref{eqn:cond-approx-J}, the upper bound in Proposition \ref{prop:cond-approx} is of the order $\mathcal{O}(h^2u^4 + h^4u^8 + h^2u^{2\alpha(x)})$. Together with the approximation error of Lemma \ref{lem:locstable-generator}, this yields the desired result.
\end{proof}

\subsection{Nonparametric inference}

\subsubsection{Ergodicity}\label{sec:ergodicity}

In this paragraph, we briefly discuss how the results of \cite{masuda2007ergodicity} may be applied in our situation to derive the geometric mixing property.
Theorem 2.2 and Lemma 2.4 therein ensure that $X_t$ given by \eqref{eqn:SDEdef} is geometrically mixing if the following assumptions hold:
\begin{enumerate}[(i)]
	\item The functions $\mu, \sigma$ are Lipschitz continuous, and the jump transformation $c(x,z)$ satisfies, for all $x_1,x_2, z_1, z_2 \in\R$, and some constant $C$, \begin{align*}
		c(x_1,0)=0,\quad |c(x_1, z_1) - c(x_2, z_1)| \leq C \,|z_1| \,|x_1-x_2|, \\ |c(x_1,z_1) - c(x_1, z_2)| \leq C|z_1-z_2|. 
	\end{align*}
	\item For some $h>0$, the transition probability $P(X_h \in dz|X_0=x)$ admits a Lebesgue density $p_h(x,z)$ which is bounded on compact sets in $x$.
	\item For every $x\in\R$ and every open set $U\subset \R$, there exists a $z\in\text{supp}(\nu)$ such that $c(x,z)\in U$.
	\item For some $q\in(0,2)$, it holds that $|\sigma(x)|\leq |x|^{1-\frac{q}{2}}$, and $\int_{|z|\geq 1} |z|^q \nu(dz) <\infty$.
	\item For some $C\in(0,\infty)$, $x\mu(x) \leq -C|x|^2$. 
\end{enumerate}

Condition (v) is crucial, as it requires the drift to induce a mean-reversion. 
The remaining conditions specify various forms of regularity. 
As the focus of the present paper is a flexible modeling of the jump part, we demonstrate that the ergodicity result of \cite{masuda2007ergodicity} also allows for a state-dependent jump activity index $\alpha(x)$.

To this end, let $\nu$ be given by $\nu(dz) = \alpha_0 |z|^{-1-\alpha_0}\, dz$ for some $\alpha_0\in(0,2)$.
Now let $c(x,z) = \text{sign}(z) \left(|z|^{\frac{\alpha}{\alpha(x)}} \wedge |z|\right)$, for a Lipschitz continuous function $x\mapsto\alpha(x)\in(0,\alpha_0)$. 
Then (i) and (iii) above are clearly satisfied, such that geometric ergodicity holds if $\sigma(x)$ and $\mu(x)$ are specified suitably.
Moreover, to compute the jump density $\rho(x,z)$, note that for $a>0$,
\begin{align*}
	\int_{c(x,z)>a} \nu(dz) = \nu\left(( a^\frac{\alpha(x)}{\alpha_0} \vee a,\infty)\right) = (a^\frac{\alpha(x)}{\alpha_0} \vee a)^{-\alpha_0} = a^{-\alpha(x)} \vee a^{-\alpha_0}.
\end{align*}
The analogous equation holds for the negative part, $a<0$, such that 
\begin{align*}
	\rho(x,z) = \begin{cases}
		\alpha(x) |z|^{-1-\alpha(x)},& |z|\leq 1 \\
		\alpha_0 |z|^{-1-\alpha_0},& |z|>1.
	\end{cases}
\end{align*}
Thus, this choice of $c(x,z)$ complies with our setup, i.e.\ (J1) and (J2).
This simple example demonstrates that the result of \cite{masuda2007ergodicity} allows for a rather flexible state-dependence of the jump intensity.

\subsubsection{Drift estimation}

The denominator $\hat{m}_n(x)$ can be shown to be consistent, using a concentration inequality for sums of dependent random variables \citep[Lemma 1.3]{bosq2012nonparametric}. In particular, under the mixing and stationarity assumptions of (K2), \cite{long2013nadaraya} use this method to establish the following result (Lemma 4.1 therein). Though they only claim convergence in probability, their proof also yields almost sure convergence by an application of the Borel-Cantelli lemma.

\begin{lemma}[\cite{long2013nadaraya}]\label{lem:kde-consistency}
	Let $G$ be a kernel satisfying (K1). If (K2) holds, and $Tb\to\infty$, then
	\begin{align*}
		\hat{m}_n(x) \overset{a.s.}{\longrightarrow} m(x) \int G(y)dy,\quad \text{as }n\to\infty.
	\end{align*}
\end{lemma}

We now proceed to establish consistency and asymptotic normality of the nonparametric drift estimator.

\begin{proof}[Proof of Theorem \ref{thm:drift-clt}]
We decompose the error of estimation into a martingale term $M_n$ and a bias term $A_n+B_n$ as \begin{align*}
	\hat{\mu}_n(x)-\mu(x) &= \frac{1}{\hat{m}_n(x)}\frac{1}{n} \sum_{i=1}^n \frac{f(X_{t_{i+1}} - X_{t_i})-E\left(f(X_{t_{i+1}} - X_{t_i})|X_{t_i}\right)}{hf'(0)} G_b(X_{t_i}-x)\\
	&\qquad + \frac{1}{\hat{m}_n(x)}\frac{1}{n} \sum_{i=1}^n \left\{\frac{E\left(f(X_{t_{i+1}} - X_{t_i})|X_{t_i}\right)}{hf'(0)} - \mu(X_{t_i})\right\} G_b(X_{t_i}-x) \\
	&\qquad + \frac{1}{\hat{m}_n(x)}\frac{1}{n} \sum_{i=1}^n \left\{\mu(X_{t_i}) - \mu(x)\right\} G_b(X_{t_i}-x) \\
	&=\frac{M_n}{\hat{m}_n(x)} + A_n + B_n.
\end{align*}
Since $\mu$ is continuously differentiable by assumption (D) and $G$ is compactly supported, we have $|B_n|\leq b \tilde{C}_{\mu,x}$. Moreover, Proposition \ref{prop:cond-approx} and Theorem \ref{thm:A2full} guarantee that \begin{align*}
	|A_n| & \leq h \sup_{|y-x|\leq b}\|\mathcal{A}^2\tau_y f\|_{\infty,0} f'(0)^{-1} \\
	&\leq hK\sum_{k=1}^4 \|\tau_x f^{(k)}\|_{\infty, (\tau\vee p_V \vee p_D)+1} 
	\quad \leq h C_f ((|x|+b) \vee 1)^{(\tau\vee p_V \vee p_D)+1}.
\end{align*}
The factor $C_f$ is finite since all derivatives of $f$ decay rapidly. Since $x$ is fixed, we obtain $A_n=\mathcal{O}(h)$.

The term $M_n$ is a martingale, and we need to control its conditional variance in order to obtain a central limit theorem. Denote $G_b^2(y)=G^2(y/b)/b$. We are led to study the term \begin{align}
	\label{eqn:driftclt-var}&\quad \frac{1}{n^2} \sum_{i=1}^n \Var\left(\frac{f(X_{t_{i+1}}-X_{t_i})}{hf'(0)}G_b(X_{t_i}-x)|X_{t_i}\right) \\
	\nonumber&= \frac{1}{b\,n^2h^2f'(0)^2} \sum_{i=1}^n \Var(f(X_{t_{i+1}}-X_{t_i})|X_{t_i})G^2_b(X_{t_i}-x)\\
	\nonumber&=\frac{1}{f'(0)^2} \frac{1}{Tb} \frac{1}{n} \sum_{i=1}^n (\mathcal{A}^*f^2(x) + a_{n,i} + a_{n,i}') G_b^2(X_{t_i}-x),
\end{align}
for $|a_{n,i}|\leq h \|\mathcal{A}^2\tau_{X_{t_i}}f^2\|_\infty + 2 h |\mathcal{A}^*f(X_{t_i})|^2 + 2 h^3 \|\mathcal{A}^2\tau_{X_{t_i}} f\|_{\infty}^2$ by Proposition \ref{prop:cond-approx}, and an additional smoothing error $|a_{n,i}'| = |\mathcal{A}^*f^2(X_{t_i}) -\mathcal{A}^*f^2(x)|$. By Theorem \ref{thm:A2full}, $|a_{n,i}|\leq h C_f$ for $|X_{t_i}-x|\leq b$. Moreover, by Lemma \ref{lem:kde-consistency}, $\frac{1}{n}\sum_{i=1}^n G^2_b(X_{t_i}-x) \to m(x) \int_{-1}^1 G^2(y)dy$ in probability under the assumptions of Theorem \ref{thm:drift-clt}. Then \begin{align*}
	&\quad \frac{1}{f'(0)^2} \frac{1}{Tb} \frac{1}{n} \sum_{i=1}^n (\mathcal{A}^*f^2(x) + a_{n,i} + a_{n,i}') G_b^2(X_{t_i}-x) \\
	&= \frac{\mathcal{A}^*f^2(x)}{Tb} \left(\frac{m(x) \int G^2}{f'(0)^2}+o_P(1) \right) + \mathcal{O}_P\left(\frac{h}{Tb}\right) + \mathcal{O}_P\left( \frac{1}{Tb} \sup_{|y-x|\leq b} \left|\mathcal{A}^*f(y)-\mathcal{A}^*f(x)\right|\right).
\end{align*}
But $\mathcal{A}^*f^2(x)=\mu(x)(f^2)'(0) + \sigma^2(x)(f^2)''(0)/2 + \mathcal{J}^*f^2(x)$ is continuously differentiable by assumptions (D), (V) and Lemma \ref{lem:genstar-smooth}, such that \begin{align*}
	 \mathcal{O}_P\left( \frac{1}{Tb} \sup_{|y-x|\leq b}\left|\mathcal{A}^*f(y)-\mathcal{A}^*f(x)\right|\right)
	 \leq \mathcal{O}_P(b/Tb)
	 = \mathcal{O}_P(1/T).
\end{align*}
The conditional variance expression \eqref{eqn:driftclt-var} is thus of the form \begin{align*}
	\eqref{eqn:driftclt-var} = \frac{m(x)\mathcal{A}^*f^2(x)}{Tb} \frac{m(x)\int G^2}{f'(0)^2}\left(1+o_P(1)\right).
\end{align*}
Hence, if $b\to 0, h\to 0, Tb\to\infty$, we have 
\begin{align*}
	\hat{\mu}_n(x)-\mu(x) = \frac{M_n}{\hat{m}_n(x)} + A_n + B_n = \mathcal{O} \left(\frac{1}{\sqrt{Tb}}\right) + \mathcal{O}(h) + \mathcal{O}(b) \pconv 0,
\end{align*}
and thus consistency. If $b,h=o\left((Tb)^{-1/2}\right)$, then the martingale part $M_n$ dominates and we can apply a standard central limit theorem for martingales to obtain asymptotic normality (e.g.\ \cite[Thm.\ 7.7.3]{durrett2010probability}). In particular, Lindeberg's condition holds because the increments of $M_n$ are bounded as \begin{align*}
	\left|\frac{f(X_{t_{i+1}}-X_{t_i}) - E\left( f(X_{t_{i+1}}-X_{t_i})|X_{t_i} \right)}{nhf'(0)} G_b(X_{t_i}-x)\right| \leq \frac{1}{Tb}\|f\|_\infty \|G\|_\infty \to 0.
\end{align*}
Hence, the martingale difference terms are bounded. 
The factor $\hat{m}_n(x)$ can be handled by Slutsky's Lemma.
\end{proof}

For the jump-filtered drift estimator, we require the following result.

\begin{lemma}\label{lem:drift-scale}
	Suppose (J1) and (J2) hold, and let $f$ satisfy (F').
	Then, as $u\to\infty$,
	\begin{align*}
		\frac{\mathcal{A}^*f^2_{u}(x)}{u^2} \to \sigma^2(x) f'(0)^2.
	\end{align*}
\end{lemma}
\begin{proof}[Proof of Lemma \ref{lem:drift-scale}]
	We have $\mathcal{A}^*f^2_u(x) = \sigma^2(x)u^2(f^2)''(0)/2 + \mathcal{J}^*f^2_u(x)$.
	Since $(f^2)''(0)=2f'(0)^2$ for our choice of $f$, it suffices to show that $\mathcal{J}^*f^2_u(x)/u^2\to 0$.
	For $g=f^2$, we write
	\begin{align*}
		\mathcal{J}^*g_u(x) 
		&= \int \left[g(uz)-g(0)-uzg'(0)\mathds{1}_{|z|\leq 1}\right] \rho(x,z)\, dz \\
		&= \int \left[g(uz)-g(0)-uzg'(0)\mathds{1}_{|z|\leq 1/u}\right] \rho(x,z)\, dz \\
		&\leq 2C_\rho\,\|g\|_{\infty} \int_{[-1,1]^c} |z|^{-1-\tau}\, dz + 4C_\rho \|g\|_\infty \int_{1/u}^{1} |z|^{-1-\alpha}\, dz \\
		&\qquad + \int_{-\frac{1}{u}}^{\frac{1}{u}} \left[g(uz)-g(0)-uzg'(0)\right]\rho(x,z)\,dz \\
		&\leq 4C_\rho \frac{\|g\|_\infty}{\tau} + 4C_\rho \|g\|_{\infty} \frac{u^\alpha}{\alpha} + \frac{1}{u}\int_{-1}^{1} \left[g(z)-g(0)-zg'(0)\right]\rho(x,z/u)\, dz.
	\end{align*}
	The last integral can be bounded as 
	\begin{align*}
		\frac{1}{u}\int_{-1}^{1} \left[g(z)-g(0)-zg'(0)\right]\rho(x,z/u)\, dz 
		&\leq \frac{C_\rho \|g''\|_\infty}{u} \int_{-1}^{1} |z|^2 |z/u|^{-1-\alpha}\, dz \\
		&\leq u^\alpha\frac{C_\rho \|g''\|_\infty}{2-\alpha}.
	\end{align*}
	That is, $\mathcal{J}^*g_u(x) = \mathcal{O}(u^\alpha)$, and $\alpha<2$ implies that $\mathcal{J}^*f^2_u(x)/u^2\to 0$ as $u\to\infty$.
	This completes the proof.
\end{proof}

\begin{proof}[Proof of Theorem \ref{thm:drift-clt-filtered}]
	Just as in the proof of Theorem \ref{thm:drift-clt} for the non-filtered estimator, we decompose $\hat{\mu}_n^*(x)-\mu(x)=M_n^*/\hat{m}_n(x) + A^*_n + B_n$ by replacing $f$ by $f_u$.
	Then $|B_n|=\mathcal{O}(b)$ as before, whereas 
	\begin{align*}
		|A_n^*| & \leq h K \sum_{k=1}^4\|\tau_x f_{u_n}^{(k)}\|_{(\tau\vee p_V\vee p_D)+1} = \mathcal{O}(hu^4).
	\end{align*}
	The last step uses that $f_u^{(k)}(x) = u^k f^{(k)}(x)$, and we bound the weighted norm via Lemma \ref{lem:normsup}.
	
	The martingale term $M_n^*$ can be studied via 
	\begin{align*}
		&\quad \frac{1}{n^2} \sum_{i=1}^n \Var\left(\frac{f_u(X_{t_{i+1}}-X_{t_i})}{hf_u'(0)} G_b(X_{t_i}-x) \Big|X_{t_i} \right) \\
		&= \frac{1}{Tb}\frac{1}{n} \frac{1}{u_n^2 f'(0)^2} \sum_{i=1}^n (\mathcal{A}^*f^2_u(x) + a_{n_i} + a'_{n,i}) G_b(X_{t_i}-x),
	\end{align*}
	where 
	\begin{align*}
		|a_{n,i}| 
		& \leq h\|\mathcal{A}^2\tau_{X_{t_i}}f^2_u\|_{\infty} + 2h |\mathcal{A}^*f(X_{t_i})|^2 + 2h^3\|\mathcal{A}^2\tau_{X_{t_i}}f\|^2_\infty, \\
		|a'_{n,i}| 
		& = |\mathcal{A}^*f_u(X_{t_i}) - \mathcal{A}^*f_u(x)|.
	\end{align*}
	From Theorem \ref{thm:A2full}, we know that $|a_{n,i}| = \mathcal{O}(hu^4)$ for $|X_{t_i}-x|\leq b$.
	Furthermore, $\frac{d}{dx} \mathcal{A}^*f^2_u(x) = \mathcal{O}(u^2)$, such that $|a'_{n,i}|=\mathcal{O}(bu^2)$ for $|X_{t_i}-x|\leq b$.
	Moreover, $\frac{1}{n} \sum_{i=1}^n G_b^2(X_{t_{i+1}}-X_{t_i}) \to m(x)\int G(y) dy$ in probability.
	Hence, 
	\begin{align*}
		 &\quad \frac{1}{Tb}\frac{1}{n} \frac{1}{u_n^2 f'(0)^2} \sum_{i=1}^n (\mathcal{A}^*f^2_u(x) + a_{n_i} + a'_{n,i}) G_b(X_{t_i}-x) \\
		 &= \frac{\mathcal{A}^*f_u^2(x)}{Tbu^2}\left( \frac{m(x)\int G^2(y)dy}{f'(0)^2} + o_P(1) \right) + \mathcal{O}_P\left( \frac{hu_n^2}{Tb} \right) + \mathcal{O}_P\left( \frac{b}{Tb} \right)\\
		 &= \frac{\sigma^2(x)}{Tb} \left( 1+o_P(1) \right),
	\end{align*} 
	where we use $hu_n^2\to 0$ and Lemma \ref{lem:drift-scale} in the last step.
	The Lindeberg condition can be checked as in the proof of \ref{thm:drift-clt}.
	Thus, $\sqrt{Tb}M_n^*/\hat{m}_n(x)$ is asymptotically normal by Slutsky's Lemma, with asymptotic variance as specified in the theorem.
	
	The additional constraints on $u_n$ ensure that $\sqrt{Tb}A_n^* = o_P(1)$ and $\sqrt{Tb}B_n = o_P(1)$, completing the proof.
\end{proof}

\begin{proof}[Proof of Theorem \ref{thm:vola-consistent}]
We decompose
\begin{align*}
	\hat{\sigma}^2_n(x) - \sigma(x) 
	&= \frac{1}{\hat{m}_n(x)}\frac{1}{n} \sum_{i=1}^n \frac{f_{u_n}(X_{t_{i+1}} - X_{t_i})-E\left(f_{u_n}(X_{t_{i+1}} - X_{t_i})|X_{t_i}\right)}{h u_n^2 f''(0)/2} G_b(X_{t_i}-x)\\
	&\qquad + \frac{1}{\hat{m}_n(x)}\frac{1}{n} \sum_{i=1}^n \left\{\frac{E\left(f_{u_n}(X_{t_{i+1}} - X_{t_i})|X_{t_i}\right)}{hu_n^2f''(0)/2} - \sigma^2(X_{t_i})\right\} G_b(X_{t_i}-x) \\
	&\qquad + \frac{1}{\hat{m}_n(x)}\frac{1}{n} \sum_{i=1}^n \left\{\sigma^2(X_{t_i}) - \sigma^2(x)\right\} G_b(X_{t_i}-x) \\
	&=\frac{M_n}{\hat{m}_n(x)} + A_n + B_n.
\end{align*}
Just as in the proof of Theorem \ref{thm:drift-clt-filtered}, we obtain that $M_n = \mathcal{O}_P(1/\sqrt{Tb})$, $B_n = \mathcal{O}_P(b)$, and $\hat{m}_n(x) \pconv m(x)$.
Regarding the term $A_n$, we use Proposition \ref{prop:cond-approx} to obtain
\begin{align*}
	E\left(f_{u_n}(X_{t_{i+1}} - X_{t_i})|X_{t_i}\right) 
	&= h\mathcal{A}^* f_{u_n}(X_{t_i}) + h^2 \mathcal{O}\left(\sup_{|y-x|\leq b} \| \mathcal{A}^2 \tau_y f \|_\infty\right) \\
	&= h u_n \mu(x) f'(0) + \frac{\sigma^2(x) u_n^2}{2} f''(0) + \mathcal{J}^*f_{u_n}(x) + h^2 u_n^4 C_f.
\end{align*}
In the last step, we applied Theorem \ref{thm:A2full}, and $C_f$ is a constant depending on $f$.
By our assumption on $f$, this yields
\begin{align*}
	A_n 
	&=\frac{1}{\hat{m}_n(x)}\frac{1}{n} \sum_{i=1}^n \left\{\frac{E\left(f_{u_n}(X_{t_{i+1}} - X_{t_i})|X_{t_i}\right)}{hu_n^2f''(0)/2} - \sigma^2(X_{t_i})\right\} G_b(X_{t_i}-x) \\
	&= \mathcal{O}_P(h\,u_n^2) + \mathcal{O}_P\left( \sup_{|y-x|\leq b} |\mathcal{J}^* f_{u_n}(y)| / u_n^2 \right).
\end{align*}

Now note that, along the same lines as in the proof of Lemma \ref{lem:drift-scale}, we may show that under assumptions (J1) and (J2), $|\mathcal{J}^*f_{u_n}(x)| \leq C_{\rho, f} u_n^\alpha$ for all $x$.
Hence, we find that $\sup_{|y-x|\leq b} \mathcal{J}^*f_{u_n}(y)/u_n^2 =\mathcal{O}(u_n^{\alpha-2})$.
Thus, if $Tb\to \infty$ and $u_n\to\infty$, we obtain $\hat{\sigma}_n^2(x)\pconv \sigma^2(x)$.
\end{proof}

\subsubsection{Jump activity estimation}

Recall the definitions
\begin{align*}
	\hat{m}_n(x) &= \frac{1}{n} \sum_{i=1}^n G_b(X_{t_i}-x) \\
	\hat{R}_n(x) = \hat{R}_n(x,1) &= \frac{1}{\hat{m}_n(x)}\frac{1}{n} \sum_{i=1}^n \frac{f(u_n(X_{t_{i+1}} - X_{t_i}))}{hu^{\alpha(x)}} G_b(X_{t_i}-x)
\end{align*}

The investigation of $\hat{\alpha}_n(x,\gamma)$ in the stable-like case is based on the following theorem, the proof of which is presented in the sequel. As a consequence, we will obtain Theorem \ref{thm:alpha-clt} by an application of the delta method.
Throughout this section, we will employ the shorthand $\hat{\alpha}_n(x) = \hat{\alpha}_n(x,\gamma)$ and $\hat{R}^*_n(x)=\hat{R}^*_n(x,\gamma)$.

\begin{theorem}\label{thm:kernel-clt}
	Let (J1), (D), (V), (SL1), (SL2) and (K1)-(K3) hold, with $\psi_n=\sqrt{Tbu^{\alpha(x)}}$. Then for any $f$ satsfying (F),
	\begin{align*}
		\sqrt{Tb u^{\alpha(x)}}\left[ \hat{R}_n(x) - r(x)f^{[\alpha(x)]}(0) \right] \Rightarrow \mathcal{N}(0, s^2), 
	\end{align*}
	with asymptotic variance $s^2=\frac{r(x)}{m(x)} (f^2)^{[\alpha(x)]}(0) \int G^2(y)dy$. If instead $\psi_n=o(\sqrt{Tbu^{\alpha(x)}})$ holds, then $\psi_n|\hat{R}_n(x) - r(x)f^{[\alpha(x)]}(0)|\pconv 0$ as $n\to\infty$. 
\end{theorem}

To study the bias of the Nadaraya-Watson estimator $\hat{R}_n(x)$, it is crucial to control the smoothness of the target function. This is partly addressed by the smoothness assumptions imposed by (SL2). A missing part is given by the following Lemma.

\begin{lemma}\label{lem:frac-continuous}
	For a twice differentiable function $f$ such that $f$ and $f''$ are bounded, the mapping $\alpha\mapsto f^{[\alpha]}(0)$ is continuously differentiable on $(0,2)$. If $f$ vanishes on $[-\zeta,\zeta]$, $\zeta>0$, then $\alpha\mapsto f^{[\alpha]}(0)$ is continuously differentiable on $(0,\infty)$.
\end{lemma}
\begin{proof}[Proof of Lemma \ref{lem:frac-continuous}]
	Assume without loss of generality that $\zeta=1$. Note that
	\begin{align*}
		\left|\frac{d}{d\alpha} f^{[\alpha]}(0)\right|&= \left|\int \left[f(z)-f(0) - z f'(0)\mathds{1}_{|z|\leq 1}\right] \frac{d}{d\alpha} |z|^{-1-\alpha}dz\right|\\
		&=\left|\int \left[f(z)-f(0) - z f'(0)\mathds{1}_{|z|\leq 1}\right] |z|^{-1-\alpha}\ln|z|\,dz \right| \\
		&\leq C\left( \|f''\|_\infty \int_{-1}^{1} |z|^{1-\alpha-\epsilon}dz + 2 \|f\|_\infty \int_1^\infty |z|^{-1-\alpha+\epsilon}dz\right),
	\end{align*}
	for $\epsilon>0$ arbitrarily small  and a factor $C=C(\epsilon)$. For $\alpha\in(0,2)$, we can take $\epsilon$ sufficiently small such that the latter integrals are finite. Thus, an integrable majorant exists. This yields continuous differentiability. If $f$ vanishes on $[-\zeta,\zeta]$, the upper bound reads as \begin{align*}
		\left|\frac{d}{d\alpha} f^{[\alpha]}(0)\right|& \leq C \|f\|_\infty \int_\zeta^\infty |z|^{-1-\alpha+\epsilon}dz,
	\end{align*}
	which is also finite in the case $\alpha\geq 2$.
\end{proof}

We are now able to prove the asymptotic result of Theorem \ref{thm:kernel-clt}.

\begin{proof}[Proof of Theorem \ref{thm:kernel-clt}]
	We introduce the martingale differences \begin{align*}
		\epsilon_i^n = \frac{\sqrt{b}}{\sqrt{Tu^{\alpha(x)}}}\left\{f(u_n(X_{t_{i+1}}-X_{t_i})) - E\left[f(u_n(X_{t_{i+1}}-X_{t_i}))|\F_{t_i}\right]\right\} G_b(X_{t_i}-x),
	\end{align*}
	such that \begin{align}
		\nonumber&\quad \hat{m}_n(x)\sqrt{Tb u^{\alpha(x)}}\left[ \hat{R}_n(x) - r(x)f^{[\alpha(x)]}(0) \right] \\
		\label{eqn:bias-kernel}&= \sqrt{Tbu^{\alpha(x)}} \frac{1}{n} \sum_{i=1}^n  \left\{\frac{E\left[ f(u_n(X_{t_{i+1}} - X_{t_i})) | \F_{t_i} \right]}{hu^{\alpha(x)}} - r(x) f^{[\alpha(x)]}(0) \right\} G_b(X_{t_i}-x) \\
		\label{eqn:martingale-kernel}&\qquad + \sum_{i=1}^n \epsilon_i^n.
	\end{align}
	We will study the martingale term \eqref{eqn:martingale-kernel} first. 
	To this end, we investigate the sum of conditional variances.
	By Theorem \ref{thm:cond-approx-SL}, there are $\F_{t_i}$ adapted r.v.s $|a_i^n| \leq \tilde{K}_f \tilde{K}_{X_{t_i}}(h^2u^4+h^4u^8 + hu^{\alpha(X_{t_i})-\delta(X_{t_i})})$ such that  \begin{align}
		\label{eqn:variance-leading}\sum_{i=1}^n \Var(\epsilon_i^n | \F_{t_i}) &= \frac{1}{n} \sum_{i=1}^n u^{\alpha(X_{t_i})-\alpha(x)} r(X_{t_i}) (f^2)^{[\alpha(X_{t_i})]}(0) \frac{1}{b} G^2\left( \frac{X_{t_i}-x}{b} \right) \\
		\label{eqn:variance-bias}&\qquad +  \frac{1}{Tu^{\alpha(x)}}\sum_{i=1}^n a_i^n \frac{1}{b} G^2\left( \frac{X_{t_i}-x}{b} \right).
	\end{align}
	Note that \begin{align}
		\nonumber& \quad \left|\frac{1}{Tu^{\alpha(x)}}\sum_{i=1}^n a_i^n \frac{1}{b} G^2\left( \frac{X_{t_i}-x}{b} \right) \right| \\
		\nonumber & \leq \tilde{K}_f   \frac{1}{n} \sum_{i=1}^n \tilde{K}_{X_{t_i}} G_b^2(X_{t_i}-x)\left( 2hu^{4-\alpha(x)} + u^{\alpha(X_{t_i})-\alpha(x)-\delta(X_{t_i})} \right)  \\
		\nonumber & \leq \tilde{K}_f \overline{K}_x \frac{1}{n} \sum_{i=1}^n G_b^2(X_{t_i}-x)\left( 2hu^{4-\alpha(x)} + u^{\alpha(X_{t_i})-\alpha(x)-(\delta(X_{t_i})-\delta(x))} u^{-\delta(x)} \right) .
	\end{align}
	In the last step, we can replace $\tilde{K}_{X_{t_i}}$ by some $\overline{K}_x<\infty$, since $G$ is supported on $[-1,1]$ and $x\mapsto \tilde{K}_x$ is locally bounded. This is the case because the denominator $\alpha(X_{t_i})-\delta(X_{t_i})$ of $\tilde{K}_{X_{t_i}}$ is bounded since we may assume w.l.o.g.\ that $\delta(y)\leq \frac{3}{4}\alpha(y)$ for each $y\in\R$, as smaller values of $\delta$ make assumption (SL1) less restrictive. Furthermore, the rate constraints (K3) on $u$ ensure that $hu^{4-\alpha(x)}\to 0$ and, for $|X_{t_i}-x|\leq b$, $u^{\alpha(X_{t_i})-\alpha(x)} = 1+ \mathcal{O}(b \log u) \to 1$. Analogously, $u^{\delta(x) - \delta(X_{t_i})}\to 1$. Hence, \eqref{eqn:variance-bias} can be upper bounded by $\tilde{K}_f \overline{K}_x o(\psi_n^{-1}) \frac{1}{n} \sum_{i=1}^n G_b^2(X_{t_i}-x)$. But $\frac{1}{n} \sum_{i=1}^n G_b^2(X_{t_i}-x)\to m(x) \int G^2(y)dy$ by Lemma \ref{lem:kde-consistency}, thus \eqref{eqn:variance-bias} tends to zero in probability. More precisely, the term \eqref{eqn:variance-bias} is of order $o_P(\psi_n^{-1})$.
	
	Regarding the term \eqref{eqn:variance-leading}, we have \begin{align*}
		\eqref{eqn:variance-leading} &= \frac{1}{n} \sum_{i=1}^n   G_b^2(X_{t_i}-x) \left[ r(x) (f^2)^{[\alpha(x)]}(0) + \tilde{a}_i^n \right],
	\end{align*}
	for $|\tilde{a}_i^n| \leq C(x) b \log(u)$, $C(x)>0$, due to the bounded support of $G_b^2$ and since $\alpha\mapsto f^{[\alpha]}(0)$ is continuously differentiable on the interval $(0,2)$ by Lemma \ref{lem:frac-continuous}. The latter term tends to zero since $b\log u \to 0$, such that we obtain \begin{align*}
		\sum_{i=1}^n \Var(\epsilon_i^n | \F_{t_i}) &= r(x) (f^2)^{[\alpha(x)]}(0) \frac{1}{n} \sum_{i=1}^n G_b^2(X_{t_i}) + o_P(1) \\
		&= r(x) (f^2)^{[\alpha(x)]}(0) m(x) \int G^2(y)dy + o_P(1).
	\end{align*}
	
	Lindeberg's condition for the martingale term \eqref{eqn:martingale-kernel} can be verified by simply noting that $|\epsilon_i^n| \leq \|f\|_\infty (Tu^{\alpha(x)}b)^{-1/2} \to 0$, since $Tb\to \infty$ by assumption. As a consequence, the central limit theorem for martingales yields (see e.g.\ \cite[Theorem 7.7.3]{durrett2010probability})\begin{align*}
		\sum_{i=1}^n \epsilon_i^n \wconv \mathcal{N}(0, s^2m(x)^2).
	\end{align*}
	
	It remains to study the bias term \eqref{eqn:bias-kernel}. By Theorem \ref{thm:cond-approx-SL}, we have \begin{align}
		\nonumber&\quad \left|\sqrt{Tbu^{\alpha(x)}} \frac{1}{n}\sum_{i=1}^n \left\{\frac{E\left[ f(u_n(X_{t_{i+1}} - X_{t_i})) | \F_{t_i} \right]}{hu^{\alpha(x)}} - r(x) f^{[\alpha(x)]}(0) \right\} G_b(X_{t_i}-x) \right| \\
		\label{eqn:bias-kernel-smooth}&\leq \sqrt{Tbu^{\alpha(x)}} \frac{1}{n}\sum_{i=1}^n \left| r(X_{t_i})f^{[\alpha(X_{t_i})]}(0) u^{\alpha(X_{t_i})-\alpha(x)} - r(x) f^{[\alpha(x)]}(0) \right| G_b(X_{t_i}-x) \\
		\nonumber&\qquad + \sqrt{Tbu^{\alpha(x)}}  \frac{1}{n}\sum_{i=1}^n \tilde{K}_f \tilde{K}_{X_{t_i}} \left( hu^{4-\alpha(X_{t_i})} + u^{-\delta(X_{t_i})} \right) G_b(X_{t_i}-x) u^{\alpha(X_{t_i})-\alpha(x)}.
	\end{align}
	Since $\tilde{K}_x$ and $\alpha(x)$ are continuous in $x$, analogously to the bound on \eqref{eqn:variance-bias}, we obtain \begin{align*}
		&\quad \sqrt{Tbu^{\alpha(x)}}  \frac{1}{n}\sum_{i=1}^n \tilde{K}_f \tilde{K}_{X_{t_i}} \left( hu^{4-\alpha(X_{t_i})} + u^{-\delta(X_{t_i})} \right) G_b(X_{t_i}-x)u^{\alpha(X_{t_i})-\alpha(x)}\\
		&\leq \sqrt{Tbu^{\alpha(x)}} \tilde{K}_f \overline{K}_{x} o_P\left( \psi_n^{-1} \right) \frac{1}{n}\sum_{i=1}^n  G_b(X_{t_i}-x)\\
		&= o_P\left( \frac{\sqrt{Tbu^{\alpha(x)}}}{\psi_n} \right).
	\end{align*}
	Thus, if $\psi_n = \sqrt{Tbu^{\alpha(x)}}$, the latter term vanishes.
	
	The remaining bias term \eqref{eqn:bias-kernel-smooth} is typical for kernel smoothers and depends on the smoothness of the function $\chi(y)=r(y)f^{[\alpha(y)]}(0) u^{\alpha(y)-\alpha(x)}$ in a neighborhood of $y=x$. We have, for a constant $C>0$ and $|y-x|\leq b$ small enough, \begin{align*}
		|\chi(y)-\chi(x)| & \leq |r(y)f^{[\alpha(y)]}(0) | \left| u^{\alpha(y)-\alpha(x)}-1 \right| \\
		&\qquad + |r(y)| \left|f^{[\alpha(y)]}(0) - f^{[\alpha(x)]}(0) \right| + |r(y)-r(x)| \left|f^{[\alpha(x)]}(0)\right| \\
		&\leq Cb\log u.
	\end{align*}
	Here, we used that $y\mapsto r(y)$, $y\mapsto \alpha(y)$ and $\beta\mapsto f^{[\beta]}(0)$ are continuously differentiable. Consequently, we may handle the term \eqref{eqn:bias-kernel-smooth} as \begin{align*}
		&\quad \sqrt{Tbu^{\alpha(x)}} \frac{1}{n}\sum_{i=1}^n \left| r(X_{t_i})f^{[\alpha(X_{t_i})]}(0) u^{\alpha(X_{t_i})-\alpha(x)} - r(x) f^{[\alpha(x)]}(0) \right| G_b(X_{t_i}-x) \\
		&\leq C\sqrt{Tb u^{\alpha(x)}} b\log u \frac{1}{n} \sum_{i=1}^n G_b(X_{t_i}-x) \\
		&= o_P\left( \frac{\sqrt{Tbu^{\alpha(x)}}}{\psi_n} \right).
	\end{align*}
	If $\psi_n = \sqrt{Tbu^{\alpha(x)}}$, this term tends to zero.
	
	Thus, we have shown that \eqref{eqn:bias-kernel} tends to zero in probability and \eqref{eqn:martingale-kernel} converges in distribution to a Gaussian limit, i.e.\ \begin{align*}
		\hat{m}_n(x) \sqrt{Tbu^{\alpha(x)}} \left[ \hat{R}_n(x) - r(x)f^{[\alpha(x)]}(0) \right] \wconv \mathcal{N}\left(0, r(x) (f^2)^{[\alpha(x)]}(0)m(x) \int G^2(y)dy \right).
	\end{align*}
	Since $\hat{m}_n(x)\to m(x)$ by Lemma \ref{lem:kde-consistency}, an application of Slutsky's theorem completes the proof of the central limit theorem.
	
	For consistency under the weaker rate constraints, note that the CLT for the martingale part \eqref{eqn:martingale-kernel} still holds. Thus, the martingale term is always smaller than $\sqrt{Tbu^{\alpha(x)}}^{-1} = o(\psi_n^{-1})$. On the other hand, we have found that the bias term is of order $o_P(\psi_n^{-1})$.
\end{proof}

Theorem \ref{thm:alpha-clt} can now be established by means of the delta method.

\begin{proof}[Proof of Theorem \ref{thm:alpha-clt}]
We write
\begin{align*}
	-\hat{\alpha}_n(x) \log(\gamma) &= \log \frac{\hat{R}_n(x)}{r(x) f_\gamma^{[\alpha(x)]}(0) } - \log \frac{\hat{R}_n(x,\gamma)}{r(x) f_\gamma^{[\alpha(x)]}(0) } \\
	&= -\alpha(x) \log \gamma + \log\frac{\hat{R}_n(x)}{r(x) f^{[\alpha(x)]}(0) } - \log \frac{\hat{R}_n(x,\gamma)}{r(x) f_\gamma^{[\alpha(x)]}(0) }.
\end{align*}
In combination with Theorem \ref{thm:kernel-clt}, an application of the delta method to the logarithm yields \begin{align*}
	-\log(\gamma) \left[\hat{\alpha}_n(x)-\alpha(x)\right] &= \frac{\hat{R}_n(x) - r(x)f^{[\alpha(x)]}(0)}{ r(x)f^{[\alpha(x)]}(0)} - \frac{\hat{R}_n(x,\gamma) - r(x)f_\gamma^{[\alpha(x)]}(0)}{ r(x)f_\gamma^{[\alpha(x)]}(0)} + o_P\left(\psi_n^{-1}\right).
\end{align*}
The leading term can be treated by Theorem \ref{thm:kernel-clt}, applied to the design function \begin{align*}
	\tilde{f}(y) = \frac{f(y)}{r(x)f^{[\alpha(x)]}(0)} - \frac{f_\gamma(y)}{r(x)f_\gamma^{[\alpha(x)]}(0)} = \frac{1}{r(x)f^{[\alpha(x)]}(0)} \left[ f(y) - \frac{f_\gamma(y)}{\gamma^{\alpha(x)}} \right].
\end{align*}
That is,
\begin{align*}
	\tilde{R}_n(x)
	&= \frac{\hat{R}_n(x) - r(x)f^{[\alpha(x)]}(0)}{ r(x)f^{[\alpha(x)]}(0)} - \frac{\hat{R}_n(x,\gamma) - r(x)f_\gamma^{[\alpha(x)]}(0)}{ r(x)f_\gamma^{[\alpha(x)]}(0)} \\
	&= \frac{1}{\hat{m}_n(x)} \frac{1}{n} \sum_{i=1}^n \frac{\tilde{f}(u_n(X_{t_{i+1}}-X_{t_i}))}{hu^{\alpha(x)}} G_b(X_{t_i}-x) \wconv \mathcal{N}(0, ),
\end{align*}
which is asymptotically normal after appropriate rescaling as in Theorem \ref{thm:kernel-clt}, i.e.\ $\sqrt{Tbu^{\alpha(x)}} \tilde{R}_n(x) \wconv \mathcal{N}(0, \tilde{s}^2)$, with asymptotic variance $\frac{r(x)}{m(x)}\tilde{s}^2=(\tilde{f}^2)^{[\alpha(x)]}(0) \int G^2(y)dy$.
This decomposition of $\hat{\alpha}_n(x)-\alpha(x)$ also yields consistency at rate $\psi_n$.

We now turn to the asymptotics of $\hat{R}^*_n(x)$. Note that
\begin{align*}
	\hat{R}_n^*(x) -r(x) &= \left[ u^{\hat{\alpha}_n(x) - \alpha(x)}-1 \right] \frac{\hat{R}_n(x)}{f^{[\hat{\alpha}_n(x)]}(0)} + \left[\frac{\hat{R}_n(x)}{f^{[\hat{\alpha}_n(x)]}(0)} - r(x)  \right].
\end{align*}
The second term can be handled by Theorem \ref{thm:kernel-clt} and the already established central limit theorem for $\hat{\alpha}_n(x)$, which yields \begin{align*}
	\left|\frac{\hat{R}_n(x)}{f^{[\hat{\alpha}_n(x)]}(0)} - r(x) \right| &\leq \frac{1}{f^{[\hat\alpha_n(x)]}(0)} \left| \hat{R}_n(x) - r(x) f^{[\alpha(x)]}(0) \right| \\
	&\qquad + \frac{r(x)}{f^{[\hat\alpha_n(x)]}(0)}  \left| f^{[\alpha(x)]}(0) - f^{[\hat\alpha_n(x)]}(0) \right| \\
	&= \mathcal{O}_P\left(\psi_n^{-1}\right),
\end{align*}
since $\alpha\mapsto f^{[\alpha]}(0)$ is continuously differentiable. As a consequence, $\hat{R}_n(x) / f^{[\hat\alpha_n(x)]}(0) \to r(x)$ in probability as $n\to\infty$. We thus study the asymptotics of the factor \begin{align*}
	 u^{\hat{\alpha}_n(x)-\alpha(x)}-1  &=  \exp\left[\left({\hat{\alpha}_n(x)-\alpha(x)}\right)\log u\right] -1.
\end{align*}
Since $|\hat{\alpha}_n(x) - \alpha(x)| \log u = \mathcal{O}_P(\psi_n^{-1} \log u) \to 0$ in probability, the delta method yields \begin{align*}
		\frac{\sqrt{Tbu^{\alpha(x)}}}{\log u} \left\{ u^{\hat{\alpha}_n(x)-\alpha(x)}-1 \right\} = \sqrt{Tbu^{\alpha(x)}} \left\{ \hat{\alpha}_n(x) - \alpha(x) \right\} + o_P(1),
\end{align*} 
which converges in distribution. Including the asymptotic scaling factor $r(x)$ completes the proof of the central limit theorem. The delta method also yields consistency at rate $\psi_n/\log u$ under the weaker conditions.
\end{proof}

\section*{Acknowledgments}
The author would like to thank two anonymous referees for constructive comments which lead to improvements of the present article, Ansgar Steland and Mia Kornely for helpful discussions, as well as the participants of the SPA 2018 conference in Gothenburg and the DYNSTOCH 2018 workshop in Porto. Declarations of interest: none

\bibliography{nonparametric-ergodic}

\begin{thebibliography}{}

\bibitem[A{\"{i}}t-Sahalia and Jacod, 2009]{ait2009estimating}
A{\"{i}}t-Sahalia, Y. and Jacod, J. (2009).
\newblock {Estimating the degree of activity of jumps in high frequency data}.
\newblock {\em The Annals of Statistics}, 37(5A):2202--2244.

\bibitem[A{\"{i}}t-Sahalia and Jacod, 2012]{ait2012identifying}
A{\"{i}}t-Sahalia, Y. and Jacod, J. (2012).
\newblock {Identifying the successive Blumenthal–Getoor indices of a
  discretely observed process}.
\newblock {\em The Annals of Statistics}, 40(3):1430--1464.

\bibitem[A{\"{i}}t-Sahalia and Jacod, 2014]{ait2014high}
A{\"{i}}t-Sahalia, Y. and Jacod, J. (2014).
\newblock {\em {High-Frequency Financial Econometrics}}.
\newblock Princeton University Press, Princeton.

\bibitem[Amorino and Gloter, 2019]{Amorino2018}
Amorino, C. and Gloter, A. (2019).
\newblock {Contrast function estimation for the drift parameter of ergodic jump
  diffusion process}.
\newblock {\em Scandinavian Journal of Statistics}.
\newblock to appear.

\bibitem[Amorino and Gloter, 2020]{Amorino2019}
Amorino, C. and Gloter, A. (2020).
\newblock {Unbiased truncated quadratic variation for volatility estimation in
  jump diffusion processes}.
\newblock {\em Stochastic Processes and their Applications}, to appear.

\bibitem[Applebaum, 2004]{applebaum2009levy}
Applebaum, D. (2004).
\newblock {\em {L{\'{e}}vy Processes and Stochastic Calculus}}.
\newblock Cambridge University Press, Cambridge.

\bibitem[Bandi and Nguyen, 2003]{bandi2003functional}
Bandi, F.~M. and Nguyen, T.~H. (2003).
\newblock {On the functional estimation of jump–diffusion models}.
\newblock {\em Journal of Econometrics}, 116(1-2):293--328.

\bibitem[Bandi and Phillips, 2003]{bandi2003fully}
Bandi, F.~M. and Phillips, P. C.~B. (2003).
\newblock {Fully nonparametric estimation of scalar diffusion models}.
\newblock {\em Econometrica}, 71(1):241--283.

\bibitem[Bass, 1988a]{bass1988occupation}
Bass, R.~F. (1988a).
\newblock {Occupation time densities for stable-like processes and other pure
  jump Markov processes}.
\newblock {\em Stochastic Processes and their Applications}, 29(1):65--83.

\bibitem[Bass, 1988b]{bass1988uniqueness}
Bass, R.~F. (1988b).
\newblock {Uniqueness in law for pure jump Markov processes}.
\newblock {\em Probability Theory and Related Fields}, 79(2):271--287.

\bibitem[Blumenthal and Getoor, 1961]{blumenthal1961sample}
Blumenthal, R.~M. and Getoor, R.~K. (1961).
\newblock {Sample functions of stochastic processes with stationary independent
  increments}.
\newblock {\em Journal of Mathematics and Mechanics}, 10(3):493--516.

\bibitem[Bosq, 1996]{bosq2012nonparametric}
Bosq, D. (1996).
\newblock {\em {Nonparametric Statistics for Stochastic Processes}}, volume 110
  of {\em Lecture Notes in Statistics}.
\newblock Springer US, New York, NY.

\bibitem[B{\"{u}}cher et~al., 2017]{Bucher2017}
B{\"{u}}cher, A., Hoffmann, M., Vetter, M., and Dette, H. (2017).
\newblock {Nonparametric tests for detecting breaks in the jump behaviour of a
  time-continuous process}.
\newblock {\em Bernoulli}, 23(2):1335--1364.

\bibitem[Bull, 2016]{bull2016near}
Bull, A.~D. (2016).
\newblock {Near-optimal estimation of jump activity in semimartingales}.
\newblock {\em The Annals of Statistics}, 44(1):58--86.

\bibitem[Durrett, 2005]{durrett2010probability}
Durrett, R. (2005).
\newblock {\em {Probability: Theory and Examples}}.
\newblock Brooks/Cole - Thomson Learning, Belmont, USA.

\bibitem[Ethier and Kurtz, 2005]{ethier2009markov}
Ethier, S.~N. and Kurtz, T.~G. (2005).
\newblock {\em {Markov Processes: Characterization and Convergence}}.
\newblock John Wiley \& Sons, Hoboken, New Jersey.

\bibitem[Funke, 2015]{Funke2015}
Funke, B. (2015).
\newblock {\em {Kernel Based Nonparametric Coefficient Estimation in Diffusion
  Models}}.
\newblock PhD thesis, TU Dortmund.

\bibitem[Gloter et~al., 2018]{gloter2018jump}
Gloter, A., Loukianova, D., and Mai, H. (2018).
\newblock {Jump filtering and efficient drift estimation for L{\'{e}}vy-driven
  SDEs}.
\newblock {\em The Annals of Statistics}, 46(4):1445--1480.

\bibitem[Hanif, 2012]{hanif2012local}
Hanif, M. (2012).
\newblock {Local linear estimation of recurrent jump-diffusion models}.
\newblock {\em Communications in Statistics - Theory and Methods},
  41(22):4142--4163.

\bibitem[Hoffmann and Dette, 2018]{Hoffmann2018a}
Hoffmann, M. and Dette, H. (2018).
\newblock {On detecting changes in the jumps of arbitrary size of a
  time-continuous stochastic process}.
\newblock {\em arXiv preprint}, 1802.08658.

\bibitem[Hoffmann et~al., 2018]{Hoffmann2018}
Hoffmann, M., Vetter, M., and Dette, H. (2018).
\newblock {Nonparametric inference of gradual changes in the jump behaviour of
  time-continuous processes}.
\newblock {\em Stochastic Processes and their Applications},
  128(11):3679--3723.

\bibitem[Jacod and Reiss, 2014]{jacod2014remark}
Jacod, J. and Reiss, M. (2014).
\newblock {A remark on the rates of convergence for integrated volatility
  estimation in the presence of jumps}.
\newblock {\em The Annals of Statistics}, 42(3):1131--1144.

\bibitem[Jacod and Todorov, 2014]{jacod2014efficient}
Jacod, J. and Todorov, V. (2014).
\newblock {Efficient estimation of integrated volatility in presence of
  infinite variation jumps}.
\newblock {\em The Annals of Statistics}, 42(3):1029--1069.

\bibitem[Jing et~al., 2012]{jing2012jump}
Jing, B.-Y., Kong, X.-B., Liu, Z., and Mykland, P. (2012).
\newblock {On the jump activity index for semimartingales}.
\newblock {\em Journal of Econometrics}, 166(2):213--223.

\bibitem[Johannes, 2004]{johannes2004statistical}
Johannes, M. (2004).
\newblock {The statistical and economic role of jumps in continuous-time
  interest rate models}.
\newblock {\em The Journal of Finance}, 59(1):227--260.

\bibitem[Lin et~al., 2014]{lin2014local}
Lin, Z., Song, Y., and Yi, J. (2014).
\newblock {Local linear estimator for stochastic differential equations driven
  by $\alpha$-stable L{\'{e}}vy motions}.
\newblock {\em Science China Mathematics}, 57(3):609--626.

\bibitem[Long and Qian, 2013]{long2013nadaraya}
Long, H. and Qian, L. (2013).
\newblock {Nadaraya-Watson estimator for stochastic processes driven by stable
  L{\'{e}}vy motions}.
\newblock {\em Electronic Journal of Statistics}, 7:1387--1418.

\bibitem[Mancini and Reno, 2011]{Mancini2011}
Mancini, C. and Reno, R. (2011).
\newblock {Threshold estimation of Markov models with jumps and interest rate
  modeling}.
\newblock {\em Journal of Econometrics}, 160(1):77--92.

\bibitem[Masuda, 2007]{masuda2007ergodicity}
Masuda, H. (2007).
\newblock {Ergodicity and exponential $\beta$-mixing bounds for
  multidimensional diffusions with jumps}.
\newblock {\em Stochastic Processes and their Applications}, 117(1):35--56.

\bibitem[Masuda, 2019]{masuda2016non}
Masuda, H. (2019).
\newblock {Non-Gaussian quasi-likelihood estimation of SDE driven by locally
  stable L{\'{e}}vy process}.
\newblock {\em Stochastic Processes and their Applications}, 129(3):1013--1059.

\bibitem[Mies, 2019]{Mies2019}
Mies, F. (2019).
\newblock {Rate-optimal estimation of the Blumenthal-Getoor index of a
  L{\'{e}}vy process}.
\newblock {\em arXiv preprint}, 1906.08062.

\bibitem[Mies and Steland, 2019]{mies2018tempering}
Mies, F. and Steland, A. (2019).
\newblock {Nonparametric Gaussian inference for stable processes}.
\newblock {\em Statistical Inference for Stochastic Processes},
  22(3):525–555.

\bibitem[Rei{\ss}, 2013]{reiss2013testing}
Rei{\ss}, M. (2013).
\newblock {Testing the characteristics of a L{\'{e}}vy process}.
\newblock {\em Stochastic Processes and their Applications}, 123(7):2808--2828.

\bibitem[Schmisser, 2014]{schmisser2014non}
Schmisser, {\'{E}}. (2014).
\newblock {Non-parametric adaptive estimation of the drift for a jump diffusion
  process}.
\newblock {\em Stochastic Processes and their Applications}, 124(1):883--914.

\bibitem[Todorov, 2015]{todorov2015jump}
Todorov, V. (2015).
\newblock {Jump activity estimation for pure-jump semimartingales via
  self-normalized statistics}.
\newblock {\em The Annals of Statistics}, 43(4):1831--1864.

\bibitem[Todorov, 2017]{todorov2017testing}
Todorov, V. (2017).
\newblock {Testing for time-varying jump activity for pure jump
  semimartingales}.
\newblock {\em The Annals of Statistics}, 45(3):1284--1311.

\bibitem[Ueltzh{\"{o}}fer, 2013]{Ueltzhofer2013}
Ueltzh{\"{o}}fer, F.~A. (2013).
\newblock {On non-parametric estimation of the L{\'{e}}vy kernel of Markov
  processes}.
\newblock {\em Stochastic Processes and their Applications},
  123(10):3663--3709.

\bibitem[Wang and Zhang, 2013]{wang2013local}
Wang, Y. and Zhang, L. (2013).
\newblock {Local linear estimation for stochastic processes driven by
  $\alpha$-stable L{\'{e}}vy motion}.
\newblock {\em Statistical Inference for Stochastic Processes}, 16(2):161--171.

\end{thebibliography}
\bibliographystyle{apalike}

\end{document}